\documentclass[10pt,draft]{amsart}
\usepackage{amsmath}
\usepackage{amscd}
\usepackage{amssymb}
\usepackage[all]{xy}
\makeindex

\newtheorem{thm}[equation]{Theorem}
\newtheorem{prop}[equation]{Proposition}
\newtheorem{cor}[equation]{Corollary}
\newtheorem{lem}[equation]{Lemma}
\theoremstyle{definition}
\newtheorem{defn}[equation]{Definition}
\newtheorem{rem}[equation]{Remark}
\newtheorem{exmp}[equation]{Example}

\numberwithin{equation}{section}

\newcommand{\arr}{\rightarrow}

\newcommand{\xarr}{\xrightarrow}

\newcommand{\cat}[1]{\operatorname{\mathsf{#1}}}

\newcommand{\opn}{\operatorname}
\newcommand{\inj}{\hookrightarrow}
\newcommand{\rmitem}[1]{\item[\text{\textup{(#1)}}]}

\newcommand{\mcal}[1]{\mathcal{#1}}

\newcommand{\mrm}[1]{\mathrm{#1}}

\newcommand{\bsym}[1]{\boldsymbol{#1}}

\newcommand{\Hom}{\operatorname{Hom}}
\newcommand{\Ext}{\operatorname{Ext}}

\newcommand{\Proj}{\operatorname{\mathsf{Proj}}}
\newcommand{\Inj}{\operatorname{\mathsf{Inj}}}
\newcommand{\Pj}{\operatorname{\mathsf{Proj}}}
\newcommand{\pj}{\operatorname{\mathsf{proj}}}
\newcommand{\Ij}{\operatorname{\mathsf{Inj}}}
\newcommand{\Mod}{\operatorname{\mathsf{Mod}}}
\newcommand{\fmod}{\operatorname{\mathsf{mod}}}

\newcommand{\Add}{\operatorname{\mathsf{Add}}}
\newcommand{\ten }{\otimes}

\def\mapdown#1{\Big\downarrow \rlap{$\vcenter{\hbox{$\scriptstyle#1$}}$}}

\newcommand{\LCM}{\operatorname{\mathsf{LCM}}}
\newcommand{\CM}{\operatorname{\mathsf{CM}}}
\newcommand{\Morph}{\operatorname{\mathsf{Mor}}}

\title[Recollement of homotopy categories and CM modules]
{Recollement of homotopy categories and Cohen-Macaulay modules}
\author{Osamu Iyama, Kiriko Kato and Jun-ichi Miyachi}
\address{O. Iyama: Graduate School of Mathematics, Nagoya University Chikusa-ku, Nagoya, 464-8602
Japan}
\email{iyama@math.nagoya-u.ac.jp}
\address{K. Kato:  Graduate School of Science, Osaka Prefecture University,
1-1 Gakuen-cho, Nakaku, Sakai, Osaka 599-8531, JAPAN}
\email{kiriko@mi.s.osakafu-u.ac.jp}
\address{J. Miyachi: Department of Mathematics, Tokyo Gakugei
University, Koganei-shi, Tokyo, 184-8501, Japan}
\email{miyachi@u-gakugei.ac.jp}

\begin{document}

\begin{abstract}
We study the homotopy category of unbounded complexes with bounded homologies and 
its quotient category by the homotopy category of bounded complexes.
We show the existence of a recollement of the above quotient category
and it has the homotopy category of acyclic complexes as a triangulated subcategory.
In the case of the homotopy category of finitely generated projective modules over 
an Iwanaga-Gorenstein ring,
we show that 
the above quotient category are triangle equivalent to the stable module category of 
Cohen-Macaulay $\opn{T}_2(R)$-modules.
\end{abstract}

\maketitle

\tableofcontents

\setcounter{section}{-1}
\section{Introduction}\label{intro}
Let $\mcal{A}$ be an abelian category, and $\mcal{B}$ its an additive full subcategory.
We consider the homotopy category $\cat{K}^{\infty, \mrm{b}}(\mcal{B})$ of unbounded complexes
with bounded homologies, its subcategories
\[\begin{CD}
\cat{K}^{\mrm{b}}(\mcal{B}) @>>> \cat{K}^{-, \mrm{b}}(\mcal{B}) \\
@VVV @VVV \\
\cat{K}^{+, \mrm{b}}(\mcal{B}) @>>> \cat{K}^{\infty, \mrm{b}}(\mcal{B})
\end{CD}\]
and the homotopy category $\cat{K}^{\infty, \emptyset}(\mcal{B})$ of acyclic complexes, 
that is, complexes with null homologies.
The homotopy categories $\cat{K}^{+, \mrm{b}}(\mcal{B})$ and $\cat{K}^{-, \mrm{b}}(\mcal{B})$ are basically important  
when we  study bounded derived categories using the category of injective objects and the one of projective objects.
The quotient category $(\cat{K}^{-,\mrm{b}}/\cat{K}^{\mrm{b}})(\mcal{B})$ of $\cat{K}^{-,\mrm{b}}(\mcal{B})$ by 
$\cat{K}^{\mrm{b}}(\mcal{B})$
 is  closely related to various stable categories in representation theory (\cite{H2}, \cite{Rd2}, \cite{Bu}).
The derived categories are originally defined by quotient categories of homotopy categories.
But in the situation that we actually deal with a derived category, we often realize it
as a subcategory of a homotopy category which is an opposite part of 
the homotopy subcategory of null complexes.
This structure of subcategories in a triangulated category is called a stable t-structure, a torsion pair, a semiorthogonal decomposition, Bousfield localization and so on (see e.g. \cite{Mi1}).

\begin{defn} \label{st-tprime}
Let $\mcal{D}$ be a triangulated category.
A pair $(\mcal{U}, \mcal{V})$ of full subcategories of $\mcal{D}$ is called a {\it stable t-structure} in $\mcal{D}$ provided that
\begin{enumerate}
\rmitem{a}  $\mcal{U}=\Sigma\mcal{U}$ and $\mcal{V}=\Sigma\mcal{V}$.
\rmitem{b}  $\opn{Hom}_{\mcal{D}}(\mcal{U}, \mcal{V}) = 0$.
\rmitem{c}  For every $X  \in \mcal{D}$, there exists a triangle $U \arr X \arr V \arr 
\Sigma U$
with $U \in \mcal{U}$ and $V \in \mcal{V}$.
\end{enumerate}
\end{defn}

In this case, $\mcal{U}$ and $\mcal{V}$ are full triangulated subcategories
of which the isomorphic closures are thick subcategories in $\mcal{D}$.
In this paper, using this notion we study triangulated full subcategories of $(\cat{K}^{\infty, \mrm{b}}/\cat{K}^{\mrm{b}})(\mcal{B})$.
First we show that if $\mcal{A}$ has enough injective (resp., projective) objects, then
\[\begin{aligned}
(\cat{K}^{\infty, \emptyset}(\mcal{B}),(\cat{K}^{+,\mrm{b}}/ \cat{K}^{\mrm{b}})(\mcal{B} ) )\
\text{and}\
((\cat{K}^{+,\mrm{b}}/ \cat{K}^{\mrm{b}})(\mcal{B} ) , (\cat{K}^{-,\mrm{b}}/ \cat{K}^{\mrm{b}})(\mcal{B} ) ) \\
\text{(resp.,}\ ((\cat{K}^{+,\mrm{b}}/ \cat{K}^{\mrm{b}})(\mcal{B} ) , (\cat{K}^{-,\mrm{b}}/ \cat{K}^{\mrm{b}})(\mcal{B} ) )\
\text{and}\
(\cat{K}^{-, \mrm{b}}(\mcal{B}),(\cat{K}^{\infty, \emptyset}/ \cat{K}^{\mrm{b}})(\mcal{B} ) ))
\end{aligned}\]
\noindent
are stable t-structures in 
$(\cat{K}^{\infty, \mrm{b}}/ \cat{K}^{\mrm{b}})(\mcal{B} )$,
where $\mcal{B}$ is the category $\Inj\mcal{A}$ (resp., $\Proj\mcal{A}$) of injective (resp., projective) objects (Theorem \ref{st-st1}).
This condition corresponds to TTF theory in an abelian category.
Hence we have the recollement
\[\xymatrix{
(\cat{K}^{*,\mrm{b}}/\cat{K}^{\mrm{b}})(\mcal{B}) \ar@<-1ex>[r]^{i_{*}}
& (\cat{K}^{\infty,\mrm{b}}/\cat{K}^{\mrm{b}})(\mcal{B})
\ar@<-2ex>[l]_{i^{*}} \ar@<2ex>[l]^{i^{!}} \ar@<-1ex>[r]^{j^{*}} 
& (\cat{K}^{\infty,\mrm{b}}/\cat{K}^{*,\mrm{b}})(\mcal{B})
\ar@<-2ex>[l]_{j_{!}} \ar@<2ex>[l]^{j_{*}}
}\]
\noindent
where $*= +$ (resp., $-$).
In particular, we have the following triangle equivalences.

\begin{thm} \label{thm01}
Let $\mcal{A}$ be an abelian category with enough injectives, $\mcal{A}'$ an abelian category 
with enough projectives, then we have a triangle equivalences
\[\begin{aligned}
&(\cat{K}^{\infty, \mrm{b}}/\cat{K}^{+, \mrm{b}})(\Inj\mcal{A}) \simeq 
(\cat{K}^{-, \mrm{b}}/\cat{K}^{\mrm{b}})(\Inj\mcal{A}) \simeq 
\cat{K}^{\infty, \emptyset}(\Inj\mcal{A})&\\
&(\cat{K}^{\infty, \mrm{b}}/\cat{K}^{-, \mrm{b}})(\Proj \mcal{A}') \simeq 
(\cat{K}^{+, \mrm{b}}/\cat{K}^{\mrm{b}})(\Proj \mcal{A}') \simeq 
\cat{K}^{\infty, \emptyset}(\Proj \mcal{A}')&
\end{aligned}\]
\end{thm}

Let us introduce the following key notion in this paper.

\begin{defn}
Let $\mcal{D}$ be a triangulated category, and let $\mcal{U}, \mcal{V},  \mcal{W}$
be triangulated full subcategories of $\mcal{D}$.
We call $(\mcal{U}, \mcal{V}, \mcal{W})$ a \emph{triangle of recollements} in $\mcal{D}$
if $(\mcal{U}, \mcal{V})$, $(\mcal{V}, \mcal{W})$ and $(\mcal{W}, \mcal{U})$ are stable t-structures in $\mcal{D}$.

If $(\mcal{U}, \mcal{V}, \mcal{W})$ is a triangle of recollements, 
all the relevant subcategories and quotient categories 
$\mcal{U}, \mcal{V}, \mcal{W},\mcal{D}/\mcal{U}, \mcal{D}/\mcal{V}$ 
and $\mcal{D}/\mcal{W}$ are equivalent. 
\end{defn}
We observed a triangle of recollements in a quotient homotopy category of an Iwanaga-Gorenstein ring. 
Such strong symmetry was never reported before. 
In the forthcoming paper \cite{IKM}, we will study an $n$-gon of recollement in general.

The following class of rings is the main object of this paper.

\begin{defn}
We call a ring $R$ \emph{Iwanaga-Gorenstein} if it is Noetherian
\footnote{All our results on Iwanaga-Gorenstein rings except those in Section \ref{secLCM} are
valid also for coherent rings with $\opn{idim}_RR<\infty$ and $\opn{idim}R_R<\infty$}
 with $\opn{idim}_RR<\infty$ and $\opn{idim}R_R<\infty$ \cite{Iw}.
Moreover, we say that \emph{$R$ has a two-sided injective resolution} if
there is an $R$-bimodule complex $V$ which is
an injective resolution of $R$ as right $R$-modules and as left $R$-modules.
\end{defn}

We have the following stronger statement for $\mcal{A}=\fmod R$,
which is our first main result in this paper  (Theorems \ref{st-h3}, \ref{st-h31}).

\begin{thm}
Let $R$ be an Iwanaga-Gorenstein ring.
Then we have a triangle of recollements
\[\begin{aligned}
(\cat{K}^{+,\mrm{b}}/ \cat{K}^{\mrm{b}})(\pj {R} ), 
(\cat{K}^{-,\mrm{b}}/ \cat{K}^{\mrm{b}})(\pj {R} ),
\cat{K}^{\infty, \emptyset}(\pj {R}))
\end{aligned}\]
in $(\cat{K}^{\infty,\mrm{b}}/  \cat{K}^{\mrm{b}})(\pj {R} )$.
If $R$ has a two-sided injective resolution, then we have a triangle of recollements
\[\begin{aligned}
(\cat{K}^{+,\mrm{b}}/ \cat{K}^{\mrm{b}})(\Pj {R} ), 
(\cat{K}^{-,\mrm{b}}/ \cat{K}^{\mrm{b}})(\Pj {R} ),
\cat{K}^{\infty, \emptyset}(\Pj {R}))
\end{aligned}\]
in $(\cat{K}^{\infty,\mrm{b}}/  \cat{K}^{\mrm{b}})(\Pj {R} )$.
\end{thm}

\begin{defn}
For an Iwanaga-Gorenstein ring $R$, define the category
of \emph{Cohen-Macaulay $R$-modules}\footnote{
In the representation theory of commutative rings \cite{Y} and orders \cite{CR,A2}, there is another notion of Cohen-Macaulay modules (see \cite{I}).
These two concept coincide for the Gorenstein case.} and \emph{large Cohen-Macaulay $R$-modules} by
\begin{eqnarray*}
\CM R&:=&\{X\in\fmod R\ |\ \Ext^i_R(X,R)=0\ (i>0)\},\\
\LCM R&:=&\{X\in\Mod R\ |\ \Ext^i_R(X, \Pj R)=0\ (i>0)\}.
\end{eqnarray*}
\end{defn}

In the representation theory of algebras, the category $\CM R$
is important from the viewpoint of (co)tilting theory \cite{ABr,ABu,AR1,AR2}.
It is well-known that $\CM R$ forms a Frobenius category with the subcategory $\pj {R}$ of projective-injective objects,
and the stable category $\underline{\CM} R$ forms a triangulated category \cite{H1}.
By a result of Happel \cite{H2}, there exists a triangle equivalence
\[\underline{\CM} R\simeq(\cat{K}^{-,\mrm{b}}/\cat{K}^{\mrm{b}})(\pj {R}),\]
which was shown by Rickard \cite{Rd2} for self-injective algebras and by Buchweitz \cite{Bu} for commutative Gorenstein rings.
Recently this category was studied by several authors motivated by string theory (e.g. \cite{O1}, \cite{O2}).
In this paper, 
we give a same type of description of this category.
We denote by $\opn{T}_{m}(R)$ an $m\times m$ upper triangular matrix ring over $R$.
Our second main result is the following (Theorems \ref{mth}, \ref{Bmcor}).

\begin{thm}\label{thm02}
Let $R$ be an Iwanaga-Gorenstein ring. Then so is $\opn{T}_2(R)$,
and there exists a triangle equivalence
\[\underline{\CM}\opn{T}_2(R) \simeq
(\cat{K}^{\infty,\mrm{b}}/\cat{K}^{\mrm{b}})(\pj {R}).\]
If $R$ has a two-sided injective resolution, then there exist triangle equivalences
\[\underline{\LCM}\opn{T}_2(R) \simeq
(\cat{K}^{\infty,\mrm{b}}/\cat{K}^{\mrm{b}})(\Pj {R})\simeq
(\cat{K}^{\infty,\mrm{b}}/\cat{K}^{\mrm{b}})(\Ij {R}).\]
\end{thm}

This is proved by using the above triangles of recollements in 
$(\cat{K}^{\infty,\mrm{b}}/  \cat{K}^{\mrm{b}})(\pj {R} )$ and
$(\cat{K}^{\infty,\mrm{b}}/  \cat{K}^{\mrm{b}})(\Pj {R} )$.

Moreover, we study the case of a Frobenius category $\mcal{F}$ with the additive subcategory $\mcal{P}$ 
of pro jective-injective objects.
Then we can apply results of  $\cat{K}^{\infty, \mrm{b}}(\mcal{B})$ to this case.
According to Theorem \ref{thm01}, we have triangle equivalences
\[\begin{aligned}
(\cat{K}^{+,\mrm{b}}/\cat{K}^{\mrm{b}})(\mcal{P}, \mcal{F})\simeq 
(\cat{K}^{-,\mrm{b}}/\cat{K}^{\mrm{b}})(\mcal{P}, \mcal{F})\simeq
\underline{\mcal{F}} \\
\underline{\Morph}^{\mrm{e}}(\mcal{F}) \simeq
\underline{\Morph}^{\mrm{m}}(\mcal{F}) \simeq
(\cat{K}^{\infty,\mrm{b}}/\cat{K}^{\mrm{b}})(\mcal{P}, \mcal{F}) .
\end{aligned}\]

\medskip
\noindent{\bf Conventions }
For a ring $R$, we denote by $\Mod R$ (resp., $\fmod R$) the category of right 
(resp., finitely presented right) $R$-modules, and denote by $\Pj {R}$ (resp., $\Ij R$, $\pj {R}$) 
the additive full subcategory of 
$\cat{Mod}R$ consisting of projective (resp., injective, finitely generated projective) modules.
For right (resp., left) $R$-module $M_R$ (resp., ${}_{R}N$), we denote by 
$\opn{idim}M_R$ (resp., $\opn{idim}{}_{R}N$) is the injective dimension
of $M_{R}$ (resp., ${}_{R}N$), and denote by 
$\opn{pdim}M_R$ (resp., $\opn{pdim}{}_{R}N$) is the projective dimension
of $M_{R}$ (resp., ${}_{R}N$).
For a triangulated category, we denote by $\Sigma$ the translation functor.

For an additive category $\mcal{B}$, we denote by $\cat{C}(\mcal{B})$ the category of complexes over $\mcal{B}$,
and by $\cat{K}(\mcal{B})$ (resp., $\cat{K}^{+}(\mcal{B})$, $\cat{K}^{-}(\mcal{B})$, $\cat{K}^{\mrm{b}}(\mcal{B})$)
the homotopy category of complexes 
(resp., bounded below complexes, bounded above complexes, bounded complexes) of $\mcal{B}$.
When $\mcal{B}$ is a full subcategory of an abelian category, we denote by
$\cat{K}^{\infty,\mrm{b}}(\mcal{B})$ (resp., $\cat{K}^{+,\mrm{b}}(\mcal{B})$ , $\cat{K}^{-,\mrm{b}}(\mcal{B})$), 
the full subcategory of $\cat{K}(\mcal{B})$ (resp., $\cat{K}^{+}(\mcal{B})$ , $\cat{K}^{-}(\mcal{B})$)
consisting of complexes with bounded homologies, and by
$\cat{K}^{\infty, \emptyset}(\mcal{B})$
the homotopy category of acyclic complexes, that is, complexes with null homologies.
For a complex $X = (X^{i}, d^{i})$, we define 
the following truncations:
\[\begin{aligned}
{\tau}_{\geq n}X & : \cdots \arr 0 \arr X^{n} 
\arr X^{n+1} \arr X^{n+2} \arr \cdots ,\\
{\tau}_{\leq n}X & : \cdots \arr X^{n-2} \arr X^{n-1}
\arr X^{n} \arr 0 \arr \cdots .
\end{aligned}\] 
For an object $M$ in an additive category $\mcal{B}$, 
we denote by $\Add M$ (resp., $\cat{add}M$) the full subcategory of $\mcal{B}$
consisting of objects which are isomorphic to summands of (resp., finite) coproducts of copies of $M$.

For an abelian category $\mcal{A}$, we denote by $\Inj \mcal{A}$ (resp., $\Proj\mcal{A}$)
the category of injective (resp., projective)
objects of $\mcal{A}$, and denote by $\cat{D}(\mcal{A})$ (resp., $\cat{D}^{\mrm{b}}(\mcal{A})$) 
the derived category of complexes (resp., bounded complexes) of objects of $\mcal{A}$.

\medskip
\noindent{\bf Acknowledgement }
We would like to thank Professor B. Keller and Professor H. Krause for helpful conversations and correspondences.

\section{Stable t-structures and recollements}\label{t-strecoll}

We recall the notion of (co)localizations and recollements  
and study the relationship with stable t-structures. 
This correspondence enables us to understand (co)localizations and recollements 
by way of subcategories instead of quotient categories. 
Indeed, Proposition~\ref{t-st-first} and Corollary~\ref{t-st-third} show that 
a (co)localization induces a stable t-structure, and vice versa. 
So the notion of (co)localizations is tantamount to that of stable t-structures. 
A recollement is a combination of a localization and a colocalization. 
In Proposition \ref{st-t-1second} we see that a recollement corresponds to two consecutive stable t-structures. 
Even more symmetric situation a triangle of recollements consists of three consecutive and recursive stable t-structures 
$(\mcal{U}, \mcal{V}), (\mcal{V}, \mcal{W})$, and $(\mcal{W}, \mcal{U})$. 
Obivously we obtain three recollements from a triangle of recollements in 
Proposition \ref{added}. 
Whenever we have a triangle of recollements, all the relevant subcategories and the quotient categories are equivalent. 
Moreover, we study a triangle functor between triangulated categories equipped with stable t-structures. 
It is natural to ask when a functor respects localizations (colocalizations, recollements etc.), 
in other words, whether it is commutative with quotient functors and their adjoint functors. 
The condition is simply given by stable t-structures (Lemma \ref{st-t-1second} and 
Corollary \ref{added}). 
The upshot is Proposition \ref{p20Apr30}: a triangle functor which respects 
a triangle of recollements 
is an equivalence if its restriction to some subcategory is so.

\vspace{2mm}

First we see that a (co)localization and a stable t-structure 
essentially describe the same phenomenon, 
using the similar methods for recollements \cite{BBD}. 

\begin{defn} \label{st-t0} 
Let $i^{*}:\mcal{D} \to \mcal{D}'$ be a triangle functors between 
triangulated categories $\mcal{D}$ and $\mcal{D}'$. 
\begin{enumerate}
\item If $i^{*}$ has a fully faithful right adjoint $i_{*}:\mcal{D}' \to \mcal{D}$, 
then $i^{*}$ or a diagram
$\xymatrix{
\mcal{D} \ar@<1ex>[r]^{i^*} 
& \mcal{D}'  \ar@<1ex>[l]^{i_*}}$
is called a {\it localization} of $\mcal{D}$. 
\item 
If $i^{*}$ has a fully faithful left adjoint $i_{!}:\mcal{D}' \to \mcal{D}$,
then $i^{*}$ or a diagram
$\xymatrix{
\mcal{D}' \ar@<1ex>[r]^{i_!} 
& \mcal{D}  \ar@<1ex>[l]^{i^*}}$
is called a {\it colocalization} of $\mcal{D}$. 
\end{enumerate}
\end{defn}

\begin{prop}[\cite{Mi1}] \label{t-st-first} 
Let $\mcal{D}$ be a triangulated category.
\begin{enumerate}
\item[I]
Let $(\mcal{U} , \mcal{V})$ be a stable t-structure in 
$\mcal{D}$. Then we have the following: 
\begin{enumerate}
\item[(1)] The canonical embedding functors $i_* : \mcal{U} \to \mcal{D}$ and $j_* : \mcal{V} \to \mcal{D}$ 
have a right adjoint $i^!:  \mcal{D} \to \mcal{U}$ and a left adjoint $j^*:  \mcal{D} \to \mcal{V}$ 
respectively such that $\opn{Ker}i^!=\mcal{V}$ and $\opn{Ker}j^*=\mcal{U}$.
In particular,  
$\xymatrix{
\mcal{U} \ar@<1ex>[r]^{i_*} 
& \mcal{D}  \ar@<1ex>[l]^{i^!}}$ is a colocalization and 
$\xymatrix{
\mcal{D} \ar@<1ex>[r]^{j^*} 
& \mcal{V}  \ar@<1ex>[l]^{j_*}}$
is a localization of $\mcal{D}$. 
\item[(2)] 
The adjunction arrows $i_{*}i^{!} \to \bsym{1}_{\mcal{D}}$
and $\bsym{1}_{\mcal{D}} \to j_{*}j^{*}$ give a triangle
$i_{*}i^{!}X \to X \to j_{*}j^{*}X \to \Sigma i_{*}i^{!}X$
for each $X \in \mcal{D}$. 
\item[(3)] $i^{!}$ and $j^{*}$ induce triangulated equivalences
$\mcal{D}/\mcal{V} \simeq \mcal{U}$ and $\mcal{D}/\mcal{U} \simeq \mcal{V}$ respectively. 
\end{enumerate} 
\end{enumerate}

\begin{enumerate}
\item[II]
\begin{enumerate}
\item[(4)] 
If $ \xymatrix{
\mcal{D} \ar@<1ex>[r]^{i^*} 
& \mcal{D}' \ar@<1ex>[l]^{i_*}
} $ is a localization of a triangulated category $\mcal{D}$,
then $(\opn{Ker}i^* , \opn{Im}i_* )$ is a stable t-structure in $\mcal{D}$.
In particular, $i^*$ induces an equivalence $\mcal{D}/\opn{Ker}i^*\simeq\mcal{D}'$.
\item[(5)] If 
$ \xymatrix{
\mcal{D}' \ar@<1ex>[r]^{i_!} 
& \mcal{D}  \ar@<1ex>[l]^{i^*}
} $ is a colocalization of a triangulated category $\mcal{D}$, 
then $(\opn{Im}i_! , \opn{Ker}i^* )$ is a stable t-structure in $\mcal{D}$.
In particular, $i^*$ induces an equivalence $\mcal{D}/\opn{Ker}i^*\simeq\mcal{D}'$.
\end{enumerate}
\end{enumerate}
\end{prop}

\begin{defn} 
A sequence of triangle functors between triangulated categories 
\[\xymatrix{
\mcal{D}' \ar@<1ex>[r]^{i_*}
& \mcal{D} \ar@<1ex>[l]^{i^!} \ar@<1ex>[r]^{j^*} 
& \mcal{D}''  \ar@<1ex>[l]^{j_*}
}\] is called a \emph{localization exact sequence}
if it satisfies the following: 
\begin{enumerate}
\rmitem{a} $ \xymatrix{
\mcal{D} \ar@<1ex>[r]^{j^*} 
& \mcal{D}'' \ar@<1ex>[l]^{j_*}
} $ is a localization and
$ \xymatrix{
\mcal{D}' \ar@<1ex>[r]^{i_*} 
& \mcal{D}  \ar@<1ex>[l]^{i^!}
} $ is a colocalization of $\mcal{D}$. 
\rmitem{b}
there are canonical embeddings 
$\opn{Im} i_* \inj \opn{Ker} j^*$ and $\opn{Im} j_* \inj \opn{Ker} i^!$ which
are equivalences. 
\end{enumerate} 
If this is the case, 
$(\opn{Ker} j^* ,  \opn{Im} j_* ) = ( \opn{Im} i_* , \opn{Ker} i^! )$ is 
a stable t-structure in $\mcal{D}$ by Proposition \ref{t-st-first}(4)(5). 
\end{defn}

Any (co)localization is completed as a localization exact sequence, which is 
an easy application of Proposition \ref{t-st-first}.

\begin{cor}\label{t-st-third}
\begin{enumerate}
\item If $ \xymatrix{
\mcal{D} \ar@<1ex>[r]^{j^*} 
& \mcal{D}'' \ar@<1ex>[l]^{j_*}
} $ is a localization of $\mcal{D}$, 
there exists a localization exact sequence 
$\xymatrix{
\mcal{D}' \ar@<1ex>[r]^{i_*}
& \mcal{D} \ar@<1ex>[l]^{i^!} \ar@<1ex>[r]^{j^*} 
& \mcal{D}''  \ar@<1ex>[l]^{j_*}}$
\item
If $ \xymatrix{
\mcal{D}'' \ar@<1ex>[r]^{j_!} 
& \mcal{D}  \ar@<1ex>[l]^{j^*}
} $ is a colocalization of $\mcal{D}$, 
there exists a localization exact sequence 
$\xymatrix{
\mcal{D}'' \ar@<1ex>[r]^{j_!}
& \mcal{D} \ar@<1ex>[l]^{j^*} \ar@<1ex>[r]^{i^*} 
& \mcal{D}'  \ar@<1ex>[l]^{i_*}}$
\end{enumerate}
\end{cor}

\begin{proof}
(2) is a dual of (1).

(1) By Proposition \ref{t-st-first}(4), we have a stable t-structure $(\mcal{U},\mcal{V})$
for $\mcal{U}:=\opn{Ker}j^*$ and $\mcal{V}:=\opn{Im}j_*$.
By Proposition \ref{t-st-first}(1), the inclusion functors $i_*:\mcal{U}\to\mcal{D}$
and $j_*:\mcal{V}\to\mcal{D}$
have a right adjoint $i^!:\mcal{D}\to\mcal{U}$ and a left adjoint $j^*:\mcal{D}\to\mcal{V}$ 
respectively such that $\opn{Ker}i^!=\mcal{V}$ and $\opn{Ker}j^*=\mcal{U}$.
Thus the conditions (a) and (b) are satisfied.
\end{proof}

The following proposition tells us how stable t-structures are inherited 
under subcategories and quotient categories. 

\begin{prop} \label{st-stK}
Let $\mathcal{D}$ be a triangulated category, $\mathcal{C}$ a thick
subcategory of $\mcal{D}$, and 
$Q:\mathcal{D} \to \mathcal{D}/\mathcal{C}$ the canonical quotient (\cite{Ne2}).
For a stable $t$-structure $(\mathcal{U}, \mathcal{V})$ in $\mathcal{D}$, the following are equivalent.
\begin{enumerate}
\item  $(\{Q(\mathcal{U})\}, \{Q(\mathcal{V})\})$ is a stable $t$-structure in $\mathcal{D}/\mathcal{C}$.
\item  $(\mathcal{U}\cap\mathcal{C}, \mathcal{V}\cap\mathcal{C})$ is a stable $t$-structure in 
$\mathcal{C}$.
\end{enumerate}

In particular, if $\mathcal{C}$ is a triangulated full subcategory 
of $\mathcal{U}$ (resp., $\mathcal{V}$), we can consider
that $\mathcal{V}$ (resp., $\mathcal{U}$) is a full subcategory of $\mathcal{D}/\mathcal{C}$, and
$(\mathcal{U}/\mathcal{C}, \mathcal{V})$ (resp., $(\mathcal{U}, \mathcal{V}/\mathcal{C})$)
is a stable $t$-structure in $\mathcal{D}/\mathcal{C}$.

Here $\{Q(\mathcal{U})\}$ is the full subcategory of $\mcal{D}/\mcal{C}$ consisting of objects
$Q(X)$ for $X \in \mcal{U}$.
\end{prop}

\begin{proof}
(1) $\Rightarrow$ (2).
For any $B \in \mathcal{C}$, we have a triangle
\[
A \to B \to C \to \Sigma A
\]
with $A \in \mcal{U}$, $C\in \mcal{V}$.
Then we have a triangle $Q(A) \to Q(B) \to Q(C) \to \Sigma Q(A)$.
Since $Q(B)$ is a null object in $\mcal{D}/\mcal{C}$, 
we have $Q(C) \simeq \Sigma Q(A) \in Q(\mcal{U})\cap Q(\mcal{V})$.
Therefore $Q(A)$ and $Q(C)$ are null objects, and hence
$A \in \mcal{U}\cap\mcal{C}$ and $C\in \mcal{V}\cap\mcal{C}$.
\par\noindent
(2) $\Rightarrow$ (1).
It suffices to show that $\Hom_{\mcal{D}/\mcal{C}}(Q(X), Q(Y))=0$
for any $X \in \mcal{U}$, $Y \in \mcal{V}$.
For $\overline{f} \in 
\Hom_{\mcal{D}/\mcal{C}}(Q(X), Q(Y))$, $\overline{f}$ is represented by
\[\xymatrix{
X' \ar[d]_{s} \ar[dr]^{f} \\
X & Y \\
}\]
where $X' \xarr{s} X \to M \to \Sigma X'$ is a triangle with $M \in \mcal{C}$. 
By the assumption, there exist a triangle $L \to M \to N \to \Sigma L$ such that
$L \in \mcal{U}\cap\mcal{C}$ and $N\in \mcal{V}\cap\mcal{C}$.
Since $\Hom_{\mcal{D}}(X, N)=0$, we have a morphism between triangles
\[\xymatrix{
X'' \ar[d]^{t} \ar[r]^{s'} & X \ar@{=}[d] \ar[r] & L \ar[d] \ar[r] & \Sigma X'' \ar[d]\\
X' \ar[r]^{s} & X \ar[r] & M \ar[r] & \Sigma X'
}\]  
Then we have $X'' \in \mcal{U}$, and $\overline{f}$ is represented by
\[\xymatrix{
X'' \ar[d]_{s'} \ar[dr]^{ft} \\
X & Y \\
}\]
Therefore $ft=0$,  and hence $\overline{f}=0$.
\end{proof}

\begin{rem}\label{st-stKR}
In Proposition \ref{st-stK}, the implication $(2) \Rightarrow (1)$ doesn't need
the thickness of $\mcal{C}$ (see \cite{Ne2}).
\end{rem}

Next let take us back to the notion of a recollement which consists of a localization and 
a colocalization. And we recall that a recollement corresponds to 
two consecutive stable t-structures.

\begin{defn} [\cite{BBD}] 
We call a diagram
\[\xymatrix{
\mcal{D'} \ar@<-1ex>[r]^{i_{*}}
& \mcal{D}
\ar@<2ex>[l]^{i^{!}} \ar@<-2ex>[l]_{i^{*}}  \ar@<-1ex>[r]^{j^{*}} 
& \mcal{D''}
\ar@<-2ex>[l]_{j_{!}} \ar@<2ex>[l]^{j_{*}}
}\]
of triangulated categories and functors a \emph{recollement} if it satisfies the following:
\begin{enumerate}
\item $i_*$, $j_!$, and $j_*$ are fully faithful.
\item $(i^* , i_* )$, $(i_* , i^!)$, $(j_! , j^* )$, and $(j^* , j_* )$ are adjoint pairs. 
\item there are canonical embeddings $\opn{Im} j_! \inj \opn{Ker}i^*$, $\opn{Im}i_* \inj\opn{Ker} j^*$, and \\
$\opn{Im} j_*\inj\opn{Ker}i^!$ which are equivalences.
\end{enumerate}
\end{defn}

\begin{prop} [\cite{Mi1}] \label{st-t-1second}
~
\begin{enumerate}
\item Let 
\[\xymatrix{
\mcal{D'} \ar@<-1ex>[r]^{i_{*}}
& \mcal{D}
\ar@<2ex>[l]^{i^{!}} \ar@<-2ex>[l]_{i^{*}}  \ar@<-1ex>[r]^{j^{*}} 
& \mcal{D''}
\ar@<-2ex>[l]_{j_{!}} \ar@<2ex>[l]^{j_{*}}
}\] be a recollement. Then 
$(\mcal{U}, \mcal{V} )$ and $(\mcal{V}, \mcal{W})$ 
are stable t-structures in $\mcal{D}$  
where we put $\mcal{U} = \opn{Im}j_{!}$, $\mcal{V} =\opn{Im}i_{*}$ and $\mcal{W} =\opn{Im}j_{*}$. 

\item  
Let $(\mcal{U}, \mcal{V})$ and $(\mcal{V}, \mcal{W})$ be stable t-structures in $\mcal{D}$.
Then for the canonical embedding $i_{*}:\mcal{V} \to \mcal{D}$,
there is a recollement 
\[\xymatrix{
\mcal{V} \ar@<-1ex>[r]^{i_{*}}
& \mcal{D}
\ar@<-2ex>[l]_{i^{*}} \ar@<2ex>[l]^{i^{!}} \ar@<-1ex>[r]^{j^{*}} 
& \mcal{D}/\mcal{V}
\ar@<-2ex>[l]_{j_{!}} \ar@<2ex>[l]^{j_{*}}
}\]
such that $\opn{Im}j_! = \mcal{U}$ and 
$\opn{Im}j_* =\mcal{W}$.
\end{enumerate}

In each cases, every object $X$ of ${\mcal D}$ has triangles 
\[ \begin{aligned} i_* i^! X \to X \to j_* j^* X \to \Sigma i_* i^! X , \\
j_! j^* X \to X \to i_* i^* X \to \Sigma j_! j^* X .  
\end{aligned} \]
\end{prop}

Thirdly, we introduce the notion of a triangle of recollements. 
A triangle of recollements results in strong symmetry.  
And we mention that as the name shows, it implies three recollements. 

\begin{defn}\footnote{More generally, $(\mcal{U}_1, \mcal{U}_2, \cdots ,  \mcal{U}_n)$  is called
an $n$-gon of recollements in $\mcal{D}$
if $(\mcal{U}_{i}, \mcal{U}_{i+1})$ is a stable t-structure in $\mcal{D}$ ($1 \leq i \leq n$), where
$\mcal{U}_1=\mcal{U}_{n+1}$ (see \cite{IKM}).}
Let $\mcal{D}$ be a triangulated category, and let $\mcal{U}_1, \mcal{U}_2,  \mcal{U}_3$
be triangulated full subcategories of $\mcal{D}$.
We call $(\mcal{U}_1, \mcal{U}_2, \mcal{U}_3)$ a \emph{triangle of recollements} in $\mcal{D}$
if $(\mcal{U}_1, \mcal{U}_2)$, $(\mcal{U}_2, \mcal{U}_3)$ and $(\mcal{U}_3, \mcal{U}_1)$
are stable t-structures in $\mcal{D}$.
\end{defn}

\begin{prop}\label{added}
Let $\mcal{D}$ be a triangulated category.
Then $(\mcal{U}_1, \mcal{U}_2, \mcal{U}_3)$ is a triangle of recollements in $\mcal{D}$ if and only if there is a recollement: 
\[\xymatrix{
\mcal{U}_k \ar@<-1ex>[r]^{i_{k*}} 
& \mcal{D}
\ar@<-2ex>[l]_{i_{k}^{*}} \ar@<2ex>[l]^{i_{k}^{!}} \ar@<-1ex>[r]^{j_{k}^{*}} 
& \mcal{D}/\mcal{U}_k
\ar@<-2ex>[l]_{j_{k!}} \ar@<2ex>[l]^{j_{k*}}
}\]
such that the essential image $\opn{Im}j_{k!}$ is $\mcal{U}_{k-1}$, and that
the essential image $\opn{Im}j_{k*}$ is $\mcal{U}_{k+1}$ 
for any $k \!\! \mod \! n$.
In this case all the relevant subcategories $\mcal{U}_k$ and 
the quotient categories $\mcal{D}/\mcal{U}_k$are triangle equivalent. 
\end{prop}

Finally we study triangulated functors 
that respect a given localization (or a colocalization, a recollement, etc.). 
By virtue of the relationship to stable t-structures, they are nothing but 
those which respect corresponding stable t-structures.

\begin{defn} Let $\mcal{D}_1$ and $\mcal{D}_2$ be triangulated categories and 
let $F:\mathcal{D}_1 \to \mathcal{D}_2$ be a triangulated functor. 
\begin{enumerate}
\item Let $(\mathcal{U}_n , \mathcal{V}_n)$ be a stable t-structures in $\mathcal{D}_n$ $(n=1,2)$. 
We say that \emph{$F$ sends $( \mcal{U}_1 , \mcal{V}_1 )$ to $( \mcal{U}_2 , \mcal{V}_2 )$}
if $F(\mcal{U}_1 )$ is contained in $\mcal{U}_2$ and 
$F(\mcal{V}_1 )$ is in $\mcal{V}_2$. 
\item Let $\theta _n$ be localization sequences in $\mcal{D}_n$ $\quad (n=1,2)$: 
\[\theta _n : \xymatrix{
\mcal{D}'_n \ar@<1ex>[r]^{{i_n}_*} 
& \mcal{D}_n \ar@<1ex>[l]^{{i_n}^!} \ar@<1ex>[r]^{{i_n}^*} 
& \mcal{D}''_n \ar@<1ex>[l]^{{j_n}_* }\\ 
}\] 
We say that \emph{$F$ sends $\theta _1$ to $\theta _2$ }
if there exist triangle functors $F' :\mcal{D}'_1 \to \mcal{D}'_2$ and 
$F'' :\mcal{D}''_1 \to \mcal{D}''_2$ that make all the squares in 
the following diagram commutative up to functorial isomorphism. 

\[\xymatrix{
\mcal{D}'_1 \ar@<1ex>[r]^{{i_1}_*} \ar[d]_{F'}
& \mcal{D}_1 \ar@<1ex>[l]^{{i_{1}}^!} \ar@<1ex>[r]^{{i_1}^*} \ar[d]^{{F}}
& \mcal{D}''_1  \ar@<1ex>[l]^{{j_{1}}_*} \ar[d]^{{F}''}\\
\mcal{D}'_2 \ar@<1ex>[r]^{{i_2}_*} 
& \mcal{D}_2 \ar@<1ex>[l]^{{i_2}^!} \ar@<1ex>[r]^{{i_2}^*} 
& \mcal{D}''_2 \ar@<1ex>[l]^{{j_2}_* }\\ 
}\] 

\item Let $\tau _n$ be a recollement in $\mcal{D}_n$ $(n=1,2)$:  
\[\tau _n : \xymatrix{
\mcal{D}'_n \ar@<-1ex>[r]^{{i_n}_{*}} 
& \mcal{D}_n
\ar@<-2ex>[l]_{{i_n}^{*}} \ar@<2ex>[l]^{{i_n}^{!}} \ar@<-1ex>[r]^{{j_n}^{*}} 
& \mcal{D}''_n 
\ar@<-2ex>[l]_{{j_n}_{!}} \ar@<2ex>[l]^{{j_n}_{*}} 
}\]
We say \emph{$F$ sends $\tau_1$ to $\tau _2$}  
if there exist triangle functors $F' :\mcal{D}'_1 \to \mcal{D}'_2$ and 
$F'' :\mcal{D}''_1 \to \mcal{D}''_2$ 
that make all the squares in the following diagram commutative up to functorial isomorphism.

\[\xymatrix{
\mcal{D}'_1 \ar@<-1ex>[r]^{{i_1}_{*}} \ar[d]_{{F}'}
& \mcal{D}_1
\ar@<-2ex>[l]_{{i_1}^{*}} \ar@<2ex>[l]^{{i_1}^{!}} \ar@<-1ex>[r]^{{j_1}^{*}} \ar[d]^{{F}}
& \mcal{D}''_1 
\ar@<-2ex>[l]_{{j_1}_{!}} \ar@<2ex>[l]^{{j_1}_{*}} \ar[d]^{{F}''}
\\
\mcal{D}'_2 \ar@<-1ex>[r]^{{i_2}_{*}} 
& \mcal{D}_2
\ar@<-2ex>[l]_{{i_2}^{*}} \ar@<2ex>[l]^{{i_2}^{!}} \ar@<-1ex>[r]^{{j_2}^{*}} 
& \mcal{D}''_2 
\ar@<-2ex>[l]_{{j_2}_{!}} \ar@<2ex>[l]^{{j_2}_{*}} 
}\] 

\item  Let $( \mcal{U}_n , \mcal{V}_n , \mathcal{W}_n)$  be 
a triangle of recollements in $\mathcal{D}_n$ $(n=1,2)$. 
We say that \emph{$F$ sends $( \mcal{U}_1 , \mcal{V}_1 , \mcal{W}_1 )$ to 
$( \mcal{U}_2 , \mcal{V}_2 , \mcal{W}_2 )$}
if $F(\mcal{U}_1 )$ is contained in $\mcal{U}_2$, 
$F(\mcal{V}_1 )$ is in $\mcal{V}_2$, and $F(\mcal{W}_1 )$ is in $\mcal{W}_2$.

\end{enumerate}
\end{defn}

Above definition (2) and (3) look reasonable but complicated. 
Although commutativity with inclusion functors automatically implies 
that with quotients and other adjoints. 

\begin{lem}\label{functor of stable t-structure} 
Let $\mcal{D}_1$ and $\mcal{D}_2$ be triangulated categories and 
let $(\mcal{U}_n , \mcal{V} _n )$ be stable t-structures in $\mathcal{D}_n$. 
Let \[\theta _n : 
\xymatrix{ \mcal{U}_n \ar@<1ex>[r]^{{i_{\mcal{U}_n}}_*} 
& \mcal{D}_n \ar@<1ex>[l]^{{i_{\mcal{U}_n}}^!} \ar@<1ex>[r]^{{j_{\mcal{U}_n}}^*} 
& \mcal{D}_n/\mcal{U}_n  \ar@<1ex>[l]^{{j_{\mcal{U}_n}}_*}
}\] be localization sequences with  
the canonical embedding functors ${i_{\mcal{U}_n}}_*$,  
the canonical quotient functors $j_{\mcal{U}_n}^*$. 
Let  $F:\mcal{D}_1\to \mcal{D}_2$ be a triangulated functor. 

Then the following are equivalent. 
\begin{enumerate}
\item 
$F$ sends $( \mcal{U}_1 , \mcal{V}_1 )$ to $( \mcal{U}_2 , \mcal{V}_2 )$. 
\item 
$F$ sends $\theta _1$ to $\theta _2$. 
\end{enumerate} 
If this is the case, the functors 
$F_1: \mcal{U}_1 \to \mcal{U}_2$ and $F_2 : \mcal{D}_1/ \mcal{U}_1 \to \mcal{D}_2/\mcal{U}_2$ 
are uniquely determined by $F$ as 
$F {i_{\mcal{U}_1}}_* = {i_{\mcal{U}_2}}_* F_1 $, $F_2 {j_{\mcal{U}_1}}^* =  {j_{\mcal{U}_2}}^* F$, 
$F_1 {i_{\mcal{U}_1}}^! \simeq  {i_{\mcal{U}_2}}^! F$, and $F {j_{\mcal{U}_1}}_* \simeq  {j_{\mcal{U}_2}}_* F_2$. 
\end{lem}

\begin{proof} 
(2) $\Rightarrow$ (1) is obvious. 

(1) $\Rightarrow$ (2). 
There uniquely exist functors 
$\hat{F}: \mcal{D}_1/\mcal{U}_1 \to \mcal{D}_2/\mcal{U}_2$, 
$\tilde{F}: \mcal{D}_1/\mcal{V}_1 \to \mcal{D}_2/\mcal{V}_2$ 
that make the following diagram commutative : 
\[\xymatrix{
\mcal{U}_1 \ar@<1ex>[r]^{{i_{\mcal{U}_1}}_*} \ar[d]_{{F}\vert_{\mcal{U}_1}}
& \mcal{D}_1 \ar@<1ex>[r]^{{j_{\mcal{U}_1}}^*} \ar[d]^{{F}}
& \mcal{D}_1/\mcal{U}_1  \ar[d]^{\hat{F}}\\
\mcal{U}_2 \ar@<1ex>[r]^{{i_{\mcal{U}_2}}_*} 
& \mcal{D}_2 \ar@<1ex>[r]^{{j_{\mcal{U}_2}}^*} 
& \mcal{D}_2/\mcal{U}_2  
}\] 
\[\xymatrix{
\mcal{V}_1 \ar@<1ex>[r]^{{i_{\mcal{V}_1}}_*} \ar[d]_{{F}\vert_{\mcal{V}_1}}
& \mcal{D}_1 \ar@<1ex>[r]^{{j_{\mcal{V}_1}}^*} \ar[d]^{{F}}
& \mcal{D}_1/\mcal{V}_1  \ar[d]^{\tilde{F}}\\
\mcal{V}_2 \ar@<1ex>[r]^{{i_{\mcal{V}_2}}_*} 
& \mcal{D}_2 \ar@<1ex>[r]^{{j_{\mcal{V}_2}}^*} 
& \mcal{D}_2/\mcal{V}_2  
}\] where ${i_{\mcal{V}_1}}_*$, ${i_{\mcal{V}_2}}_*$ are the canonical embedding functors and
$j_{\mcal{V}}^*$, $j_{\mcal{V}_2}^*$ the canonical quotient functors. 
On the other hand, a right adjoint $i_{\mcal{U}_1}^!$ ($i_{\mcal{U}_2}^!$ ) of 
${i_{\mcal{U}_1}}_*$ (${i_{\mcal{U}_2}}_*$) induces an equivalence 
$e: \mcal{D}_1/\mcal{V}_1 \to \mcal{U}_1$ ($e' : \mcal{D}_2/\mcal{V}_2 \to \mcal{U}_2$) 
such that $i_{\mcal{U}_1}^! = e j_{\mcal{V}_1}^*$ ($i_{\mcal{U}_2}^! = e j_{\mcal{V}_2}^*$). 
Define the functor $F_1 :\mcal{U}_1 \to \mcal{U}_2$ as 
$e' \hat{F} = F_1 e$. Then $F_1 i_{\mcal{U}_1}^! = i_{\mcal{U}_2}^! F$. 
And $F_1 = F_1 i_{\mcal{U}_1}^! {i_{\mcal{U}_1}}_* = i_{\mcal{U}_2}^! F {i_{\mcal{U}_1}}_* 
=i_{\mcal{U}_2}^! {i_{\mcal{U}_2}}_* F\vert _{\mcal{U}_1} =F\vert _{\mcal{U}_1}$. 
Similarly, using equivalences 
$g:\mcal{D}_1/\mcal{U}_1 \to \mcal{V}_1$ and $g' : \mcal{D}_2/\mcal{U}_2 \to \mcal{V}_2$ induced by 
left adjoints $i_{\mcal{V}_1}^*$, $i_{\mcal{V}_2}^*$ of 
${i_{\mcal{V}_1}}_*$, ${i_{\mcal{V}_2}}_*$ respectively, we get a functor 
$F_2 : \mcal{D}_1/\mcal{U}_1 \to \mcal{D}_2/\mcal{U}_2$ as 
$F {j_{\mcal{U}_1}}_* = {j_{\mcal{U}_2}}_* F_2$. 
The functors $F_2$ and $\tilde{F}$ coincide as well. 
\end{proof}

We have the following immediate consequence of Lemma \ref{functor of stable t-structure}.

\begin{cor}\label{p7Apr30} 
Let $\mcal{D}_1$ and $\mcal{D}_2$ be triangulated categories and 
$( \mcal{U}_n , \mcal{V}_n )$, $( \mcal{V}_n , \mcal{W}_n )$ be 
stable t-structures in $\mcal{D}_n$ ($n=1,2$). 
The following are equivalent for a triangle functor 
$F:\mcal{D}_{1} \to \mcal{D}_{2}$:
\begin{enumerate}
\item $F$ sends $( \mcal{U}_1 , \mcal{V}_1 )$ to $( \mcal{U}_2 , \mcal{V}_2 )$ and 
$( \mcal{V}_1 , \mcal{W}_1 )$ to $( \mcal{V}_2 , \mcal{W}_2 )$. 
\item $F$ sends $\tau _1$ to $\tau _2$ where $\tau _n$ is a recollement 
\[\xymatrix{
\mcal{V}_n \ar@<-1ex>[r]^{{i_n}_{*}} 
& \mcal{D}_2
\ar@<-2ex>[l]_{{i_n}^{*}} \ar@<2ex>[l]^{{i_n}^{!}} \ar@<-1ex>[r]^{{j_n}^{*}} 
& \mcal{D}_2 /\mcal{V}_2
\ar@<-2ex>[l]_{{j_n}_{!}} \ar@<2ex>[l]^{{j_n}_{*}} 
}\] 
with a canonical embedding functor ${i_n}_*$ and a canonical quotient functor ${j_n}^*$ ($n=1,2$). 
\end{enumerate}
If this is the case, triangle functors $F': \mcal{V}_1 \to \mcal{V}_2$ and 
$F'': \mcal{D}_1/ \mcal{V}_1 \to \mcal{D}_2/\mcal{V}_2$ are uniquely determined as 
$F {j_1}_! \simeq  {j_2}_! F'' $, $F' {i_1}^* \simeq  {i_2}^* F$, 
$F'' {j_1}^* = {j_2}^* F$, $F {i_1}_* = {i_2}_* F'$, 
$F {j_1}_* \simeq  {j_2}_* F'' $, and $F' {i_1}^! \simeq  {i_2}^! F$. 
\[\xymatrix{
\mcal{V}_1 \ar@<-1ex>[r]^{{i_1}_{*}} \ar[d]_{{F}'}
& \mcal{D}_1
\ar@<-2ex>[l]_{{i_1}^{*}} \ar@<2ex>[l]^{{i_1}^{!}} \ar@<-1ex>[r]^{{j_1}^{*}} \ar[d]^{{F}}
& \mcal{D}_1 /\mcal{V}_1
\ar@<-2ex>[l]_{{j_1}_{!}} \ar@<2ex>[l]^{{j_1}_{*}} \ar[d]^{{F}''}
\\
\mcal{V}_2 \ar@<-1ex>[r]^{{i_2}_{*}} 
& \mcal{D}_2
\ar@<-2ex>[l]_{{i_2}^{*}} \ar@<2ex>[l]^{{i_2}^{!}} \ar@<-1ex>[r]^{{j_2}^{*}} 
& \mcal{D}_2 /\mcal{V}_2
\ar@<-2ex>[l]_{{j_2}_{!}} \ar@<2ex>[l]^{{j_2}_{*}} 
}\]
\end{cor}

In the rest of this section we shall give a useful criterion to show that a triangle functor is an equivalence
by looking at the restriction to subcategories.
We need the following preparation.

\begin{lem}\label{17-1Apr30}
If a triangle functor $F:\mcal{D}_{1} \to \mcal{D}_{2}$ 
sends a stable t-structure $( \mcal{U}_1 , \mcal{V}_1 )$ in $\mcal{D}_1$ 
to a stable t-structure $( \mcal{U}_2 , \mcal{V}_2 )$ in $\mcal{D}_2$. 
Then we have the following: 
\begin{enumerate}
\item If $F\mid _{\mcal{U}_1 }$ is full (resp., faithful), then 
$\Hom _{\mcal{D}_1}( U, X) \to \Hom _{\mcal{D}_2}( FU, FX)$ is 
surjective (resp., injective) for 
$U\in  {\mcal{U}_1}$ and $X\in  {\mcal{D}_1}$. 
\item If $F\mid _{\mcal{V}_1 }$ is full (resp., faithful), then 
$\Hom _{\mcal{D}_1}( X, V) \to \Hom _{\mcal{D}_2}( FX, FV)$ is 
surjective (resp., injective) for 
$X\in  {\mcal{D}_1}$ and $V\in  {\mcal{V}_1}$. 
\item If $F$ is full and $F\mid _{\mcal{U}_1}$ and $F\mid_{\mcal{V}_1}$ are dense, 
then $F$ is dense. 
\end{enumerate}
\end{lem}

\begin{proof}
(1) is a dual of (2).

(2) A given $X\in \mcal{D}_1$ has triangles  
\[ U_X \to X \to V^X \to \Sigma U_X \] 
\[ FU_X \to FX \to FV^X \to \Sigma FU_X \] 
where $U_X \in \mcal{U}_1$ and $V^X \in \mcal{V}_1$. 
These induce $\Hom_{\mcal{D}_1} (X,V) \simeq \Hom_{\mcal{D}_1} (V^X,V)$ 
and $\Hom_{\mcal{D}_2 }(FX,FV) \simeq \Hom_{\mcal{D}_2} (FV^X,FV)$ for $V\in \mcal{V}_1$. 
Hence surjectivity (resp. injectivity) of $\Hom_{\mcal{D}_1} (V^X,V) \to \Hom_{\mcal{D}_2} (FV^X,FV)$ 
implies that of $\Hom_{\mcal{D}_1} (X,V) \to \Hom_{\mcal{D}_2} (FX,FV)$.

(3) Let $X'$ be an object of $\mcal{D}_2$. 
There is a triangle 
\[ U'_{X'} \to X' \to {V'}^{X'} \stackrel{f}{\to} \Sigma {U'}_{X'} \] 
with $ U'_{X'} \in \mcal{U}_2$ and $ {V'}^{X'} \in \mcal{V}_2$. 
There exist $V \in \mcal{V}_1$, $U\in \mcal{U}_1$ such that 
$FV\simeq V'$, $FU\simeq U'$. And there exists $g\in \Hom_{\mcal{D}_1} (V, \Sigma U )$ 
such that we have a morphism between triangles 
\[\begin{CD}
FU @>>> FX @>>> FV @>Fg>> \Sigma FU \\
@VV\wr V @VVV @VV\wr V @VV\wr V \\
U' @>>> X' @>>>  V' @>f>> \Sigma U'
\end{CD}\] 
We have $FX \simeq X'$. 
\end{proof}

We have the following desired criterion.

\begin{cor}\label{p17-1Apr30} 
Let $\mcal{D}_1$, $\mcal{D}_2$ be triangulated categories. 
Let $( \mcal{U}_n , \mcal{V}_n )$ and $( \mcal{V}_n , \mcal{W}_n )$ 
be stable t-structures in $\mcal{D}_n$~$(n=1,2)$. 
Assume a triangle functor $F:\mcal{D}_1 \to \mcal{D}_2$ 
sends $( \mcal{U}_1 , \mcal{V}_1 )$ and  $( \mcal{V}_1 , \mcal{W}_1 )$ to 
$( \mcal{U}_2 , \mcal{V}_2 )$ and $( \mcal{V}_2 , \mcal{W}_2 )$ respectively. 
If $F \mid_{\mcal{U}_1}$ and $F \mid _{\mcal{V}_1}$ are fully faithful (resp., equivalent), so is $F$. 
\end{cor} 

\begin{proof}
For given $X,Y \in \mcal{D}_1$, there is a triangle
\[ U_X \to X \to V^X \to \Sigma U_X \] 
where $U_X\in \mcal{U}_1$ and $V^X\in \mcal{V}_1$. 
These triangles induce a diagram of abelian groups with exact rows and columns 
{\tiny\[  \begin{array}{ccccccccc}
{}_{\mcal{D}_1}(\Sigma U_X , Y)\to&{}_{\mcal{D}_1}(V^X , Y)\to&{}_{\mcal{D}_1}(X , Y)\to&{}_{\mcal{D}_1}(U_X , Y)\to&{}_{\mcal{D}_1}(\Sigma^{-1}V^X , Y)\\
\downarrow&\downarrow&\downarrow&\downarrow&\downarrow\\
{}_{\mcal{D}_2}(F\Sigma U_X , FY)\to&{}_{\mcal{D}_2}(FV^X , FY)\to&{}_{\mcal{D}_2}(FX , FY)\to&{}_{\mcal{D}_2}(FU_X , FY)\to&{}_{\mcal{D}_2}(F\Sigma^{-1}V^X , FY)
\end{array} \]} 
where ${}_{\mcal{D}_n}(~ ,~)$ denote $\Hom_{\mcal{D}_n}(~ ,~)$. 
Since $\Hom_{\mcal{D}_1}(U , Y)\to \Hom_{\mcal{D}_2}(FU , FY)$ and
$\Hom_{\mcal{D}_1}(V , Y)\to \Hom_{\mcal{D}_2}(FV , FY)$ are 
bijective for any $U\in\mcal{U}_1$ and $V\in\mcal{V}_1$ by Lemma~\ref{17-1Apr30}(1), 
so is $\Hom_{\mcal{D}_1}(X , Y) \to \Hom_{\mcal{D}_2}(FX , FY)$ 
from five-lemma. 

In the case that $F\mid_{\mcal{U}_1}$ and $F\mid _{\mcal{V}_1}$ are equivalences, $F$ is dense by Lemma~\ref{17-1Apr30}(3).  
\end{proof}

As an immediate consequence, we have the following useful result, which will be used in the later section.

\begin{prop}\label{p20Apr30} 
Let $\mcal{D}_n$ be a triangulated category, and let $( \mcal{U}_n , \mcal{V}_n , \mcal{W}_n )$ 
be a triangle of recollements in $\mcal{D}_n$ $(n=1,2)$.
Assume a triangle functor $F:\mcal{D}_{1} \to \mcal{D}_{2}$ 
sends a triangle of recollements $(\mcal{U}_1, \mcal{V}_1,\mcal{W}_1 )$ to 
a triangle of recollements $(\mcal{U}_2,\mcal{V}_2,\\
\mcal{W}_2 )$. 
If $F \mid _{\mcal{U}_1}$ is fully faithful (resp., equivalent), so is $F:\mcal{D}_1\to\mcal{D}_2$.
\end{prop} 

\begin{proof}
By Proposition \ref{t-st-first}, Cororally~\ref{p7Apr30}, it easy to see that
$F'':\mcal{D}_1/\mcal{V}_1 \to \mcal{D}_2/\mcal{V}_2$ is  fully faithful (resp., equivalent), and
hence so is $F\mid_{\mcal{W}_1}$.
Similarly, we have that $F \mid_{\mcal{V}_1}$ is fully faithful (resp., equivalent).
Hence we have the statement by Lemma~\ref{17-1Apr30} (3).
\end{proof}

\section{Recollement of homotopy categories}\label{recollhtp}

In this section, we study recollements of homotopy categories. 
Let us start with the following general observation.

\begin{prop}\label{pm}
Let $\mcal{A}$ be an abelian category, and $\mcal{B}$ its additive subcategory.
Then $((\cat{K}^{+,\mrm{b}}/ \cat{K}^{\mrm{b}})( \mcal{B}), 
(\cat{K}^{-,\mrm{b}}/ \cat{K}^{\mrm{b}})(\mcal{B} ) )$ is a stable t-structure 
in $(\cat{K}^{\infty ,\mrm{b}}/\cat{K}^{\mrm{b}})(\mcal{B})$. 
\end{prop} 

\begin{proof}
It is clear that we have
\[
\Hom_{(\cat{K}^{\infty,\mrm{b}}/\cat{K}^{\mrm{b}})(\mcal{B})}
((\cat{K}^{+,\mrm{b}}/\cat{K}^{\mrm{b}})(\mcal{B}),
(\cat{K}^{-,\mrm{b}}/\cat{K}^{\mrm{b}})(\mcal{B}))=0 
\]
For any $X \in \cat{K}^{\infty,\mrm{b}}(\mcal{B})$,
$0 \to \tau_{\geq 1}X \to X \to \tau_{\leq 0}X \to 0$ is a short exact sequence, 
which implies a triangle 
$\tau_{\geq 1}X \to X \to \tau_{\leq 0}X \to \Sigma\tau_{\geq 1}X$. 
Since $\tau_{\geq 1}X \in \cat{K}^{+,\mrm{b}}(\mcal{B})$
and $\tau_{\leq 0}X \in \cat{K}^{-,\mrm{b}}(\mcal{B})$, we get the conclusion. 
\end{proof}

\begin{prop} \label{st-st0}
Let $\mcal{A}$ be an abelian category with enough injectives. Then we have the following.
\begin{enumerate}
\item
A pair $(\cat{K}^{\infty, \emptyset}(\Inj\mcal{A}),\cat{K}^{+,\mrm{b}}(\Inj\mcal{A}))$ 
is a stable t-structure in $\cat{K}^{\infty,\mrm{b}}(\Inj\mcal{A})$.
\item 
There is a localization exact sequence
\[\xymatrix{
\cat{K}^{\infty, \emptyset}(\Inj\mcal{A}) \ar@<1ex>[r]^{j_!}
& \cat{K}^{\infty,\mrm{b}}(\Inj\mcal{A}) \ar@<1ex>[l]^{j^*} \ar@<1ex>[r]^{\ i^*} 
& \cat{D}^{\mrm{b}}(\mcal{A})  \ar@<1ex>[l]^{\ i_*}
}\] 
for the canonical embedding 
$j_!:\cat{K}^{\infty, \emptyset}(\Inj\mcal{A}) \to \cat{K}^{\infty,\mrm{b}}(\Inj\mcal{A})$. 
\end{enumerate}
\end{prop} 

\begin{proof}
It is clear that 
$\Hom_{\cat{K}^{\infty,\mrm{b}}(\Inj\mcal{A})}(\cat{K}^{\infty, \emptyset}(\Inj\mcal{A}),
\cat{K}^{+,\mrm{b}}(\Inj\mcal{A}))=0$.
Given any complex $X \in \cat{K}^{\infty,\mrm{b}}(\Inj\mcal{A})$, 
there is a complex $V_X \in \cat{K}^{+,\mrm{b}}(\Inj\mcal{A})$ which has a quasi-isomorphism
$X \xarr{\epsilon_X} V_X$.  By taking the mapping cone of $\epsilon_X$, we have a
triangle $U_X \xarr{\eta_X} X \xarr{\epsilon_X}V_X \to \Sigma U_X$ with
$U_X \in \cat{K}^{\infty, \emptyset}(\Inj\mcal{A})$.
Then $(\cat{K}^{\infty, \emptyset}(\Inj\mcal{A}), \\
\cat{K}^{+,\mrm{b}}(\Inj\mcal{A}))$ is a stable t-structure
in $\cat{K}^{\infty,\mrm{b}}(\Inj\mcal{A})$. By Proposition \ref{t-st-first} and 
Corollary \ref{t-st-third}, we are done with the proof.  
\end{proof} 

Dually we have the following observation.

\begin{prop}\label{st-st1f}
Let $\mcal{A}'$ be an abelian category with enough projectives. 
Then we have the following. 
\begin{enumerate}
\item 
A pair $(\cat{K}^{-,\mrm{b}}(\Proj\mcal{A}'),\cat{K}^{\infty, \emptyset}(\Proj\mcal{A}'))$ 
is a stable t-structure in \\
$\cat{K}^{\infty,\mrm{b}}(\Proj\mcal{A}')$.
\item
There is a localization exact sequence 
\[\xymatrix{
\cat{D}^{\mrm{b}}(\mcal{A}') \ar@<1ex>[r]^{i_*}
& \cat{K}^{\infty,\mrm{b}}(\Proj\mcal{A}') \ar@<1ex>[l]^{i^!} \ar@<1ex>[r]^{\ j^*} 
& \cat{K}^{\infty , \emptyset} (\mcal{A}')  \ar@<1ex>[l]^{\ j_*}
}\] 
for the canonical embedding 
$j_*:\cat{K}^{\infty ,\emptyset }(\Proj\mcal{A}') \to \cat{K}^{\infty,\mrm{b}}(\Proj\mcal{A}')$,
\end{enumerate} 
\end{prop}

\begin{thm} \label{st-st1}~
\begin{enumerate}
\item
Let $\mcal{A}$ be an abelian category with enough injectives.
\begin{enumerate}
\item[(a)]
$((\cat{K}^{\infty, \emptyset}(\Inj\mcal{A}), (\cat{K}^{+,\mrm{b}}/ \cat{K}^{\mrm{b}})(\Inj\mcal{A} ))$ and 
$((\cat{K}^{+,\mrm{b}}/ \cat{K}^{\mrm{b}})(\Inj\mcal{A} ),\\ 
(\cat{K}^{-,\mrm{b}}/ \cat{K}^{\mrm{b}})(\Inj\mcal{A} ))$ 
are stable t-structures in
$(\cat{K}^{\infty, \mrm{b}}/ \cat{K}^{\mrm{b}})(\Inj\mcal{A} )$. 
\item[(b)]
The canonical embedding 
$i_*:(\cat{K}^{+,\mrm{b}}/\cat{K}^{\mrm{b}})(\Inj\mcal{A})
\to (\cat{K}^{\infty,\mrm{b}}/\cat{K}^{\mrm{b}})(\Inj\mcal{A})$
induces a recollement
{\small \[\xymatrix{
(\cat{K}^{+,\mrm{b}}/\cat{K}^{\mrm{b}})(\Inj\mcal{A}) \ar@<-1ex>[r]^{i_{*}}
& (\cat{K}^{\infty,\mrm{b}}/\cat{K}^{\mrm{b}})(\Inj\mcal{A})
\ar@<-2ex>[l]_{i^{*}} \ar@<2ex>[l]^{i^{!}} \ar@<-1ex>[r]^{j^{*}} 
& (\cat{K}^{\infty,\mrm{b}}/\cat{K}^{+,\mrm{b}})(\Inj\mcal{A})
\ar@<-2ex>[l]_{j_{!}} \ar@<2ex>[l]^{j_{*}}
}\]}
such that $\opn{Im}j_{!}=\cat{K}^{\infty, \emptyset}(\Inj\mcal{A})$ and 
$\opn{Im}j_{*}=(\cat{K}^{-, \mrm{b}}/\cat{K}^{\mrm{b}})(\Inj\mcal{A})$.
\end{enumerate}
\item 
Let $\mcal{A}'$ be an abelian category with enough projectives.
\begin{enumerate}
\item[(a)]
$((\cat{K}^{+,\mrm{b}}/ \cat{K}^{\mrm{b}})(\Proj\mcal{A}' ),
(\cat{K}^{-,\mrm{b}}/ \cat{K}^{\mrm{b}})(\Proj\mcal{A}' ))$ and 
$((\cat{K}^{-,\mrm{b}}/ \cat{K}^{\mrm{b}})(\Proj\mcal{A}' ), \\
\cat{K}^{\infty, \emptyset}(\Proj\mcal{A}' ) )$ are 
stable t-structures in
$(\cat{K}^{\infty,\mrm{b}}/  \cat{K}^{\mrm{b}})(\Proj\mcal{A}' )$. 
\item[(b)]
The canonical embedding 
$i_*:(\cat{K}^{-,\mrm{b}}/\cat{K}^\mrm{b})(\Proj\mcal{A}')
\to \\
(\cat{K}^{\infty,\mrm{b}}/\cat{K}^\mrm{b})(\Proj\mcal{A}')$
induces a recollement
{\footnotesize
\[\xymatrix{
(\cat{K}^{-,\mrm{b}}/\cat{K}^\mrm{b})(\Proj\mcal{A}') \ar@<-1ex>[r]^{i_{*}}
& (\cat{K}^{\infty,\mrm{b}}/\cat{K}^\mrm{b})(\Proj\mcal{A}')
\ar@<-2ex>[l]_{i^{*}} \ar@<2ex>[l]^{i^{!}} \ar@<-1ex>[r]^{j^{*}} 
& (\cat{K}^{\infty,\mrm{b}}/\cat{K}^{-,\mrm{b}})(\Proj\mcal{A}')
\ar@<-2ex>[l]_{j_{!}} \ar@<2ex>[l]^{j_{*}}
}\] }
such that
$\opn{Im}j_{!}=(\cat{K}^{-, \mrm{b}}/\cat{K}^\mrm{b})(\Proj\mcal{A}')$ and
$\opn{Im}j_{*}=\cat{K}^{\infty, \emptyset}(\Proj\mcal{A}')$.
\end{enumerate}

\end{enumerate}
\end{thm}

\begin{proof} 
Immediately from Propositions ~\ref{pm}, \ref{st-st0}, \ref{st-st1f}, \ref{st-stK}.
\end{proof}

\begin{cor} \label{st-h2}
\begin{enumerate}
\item Let $\mcal{A}$ be an abelian category with enough injectives.
Then we have a triangle equivalence
$(\cat{K}^{-, \mrm{b}}/\cat{K}^{\mrm{b}})(\Inj\mcal{A}) \simeq 
\cat{K}^{\infty, \emptyset}(\Inj\mcal{A})$.
\item Let $\mcal{A}'$ be an abelian category with enough projectives.
Then we have a triangle equivalence
$(\cat{K}^{+, \mrm{b}}/\cat{K}^{\mrm{b}})(\Proj \mcal{A}') \simeq 
\cat{K}^{\infty, \emptyset}(\Proj \mcal{A}')$.
\end{enumerate}
\end{cor}

\begin{proof}
(1) According to Theorem \ref{st-st1}, we have the triangle equivalences \\
$\cat{K}^{\infty,\emptyset}(\Inj\mcal{A})\simeq \opn{Im}j_{!} \simeq 
(\cat{K}^{\infty,\mrm{b}}/\cat{K}^{+,\mrm{b}})(\Inj\mcal{A}) \simeq 
\opn{Im}j_{*}\simeq (\cat{K}^{-,\mrm{b}}/\cat{K}^{\mrm{b}})(\Inj\mcal{A})$. 
Similarly we get (2). 
\end{proof}

For Iwanaga-Gorenstein rings, we have one more stable t-structure.

\begin{lem}\label{st-h4}
Let $R$ be an Iwanaga-Gorenstein ring.
Then 
$(\cat{K}^{\infty ,\emptyset}(\pj {R}), \\
\cat{K}^{+,\mrm{b}}(\pj {R}))$ is 
a stable t-structure in $\cat{K}^{\infty ,\mrm{b}}(\pj {R})$.
\end{lem}

\begin{proof}
By the assumption, $\Hom_{R}(-,R)$ induces dualities 
\[\begin{aligned}
\cat{K}^{\infty,\mrm{b}}(\pj {R}) \simeq \cat{K}^{\infty,\mrm{b}}(\pj {R^{\mrm{op}}}),\ \ \ &
\cat{K}^{\infty, \emptyset}(\pj {R}) \simeq \cat{K}^{\infty, \emptyset}(\pj {R^{\mrm{op}}}) \\
\cat{K}^{-, \mrm{b}}(\pj {R}) \simeq \cat{K}^{+, \mrm{b}}(\pj {R^{\mrm{op}}}),\ \ \ &
\cat{K}^{+, \mrm{b}}(\pj {R}) \simeq \cat{K}^{-, \mrm{b}}(\pj {R^{\mrm{op}}}) .
\end{aligned}\]
Hence a stable t-structure 
$( \cat{K}^{-, \mrm{b}}(\pj {R^{\mrm{op}}}), \cat{K}^{\infty, \emptyset}(\pj {R^{\mrm{op}}}))$ 
in $\cat{K}^{\infty,\mrm{b}}(\pj {R^{\mrm{op}}})$ implies 
a stable t-structure $(\cat{K}^{\infty ,\emptyset}(\pj {R}), \cat{K}^{+,\mrm{b}}(\pj {R}))$
in $\cat{K}^{\infty ,\mrm{b}}(\pj {R})$.
\end{proof}

\begin{thm} \label{st-h3}
Let $R$ be an Iwanaga-Gorenstein ring, 
Then
\[
((\cat{K}^{+,\mrm{b}}/\cat{K}^{\mrm{b}})(\pj {R}), 
(\cat{K}^{-,\mrm{b}}/\cat{K}^{\mrm{b}})(\pj {R}), 
\cat{K}^{\infty ,\emptyset}(\pj {R}))
\] 
is a triangle of recollements
in $(\cat{K}^{\infty ,\mrm{b}}/\cat{K}^{\mrm{b}})(\pj {R})$.
In particular, we have triangle equivalences
\[
(\cat{K}^{+,\mrm{b}}/\cat{K}^{\mrm{b}})(\pj {R})\simeq 
(\cat{K}^{-,\mrm{b}}/\cat{K}^{\mrm{b}})(\pj {R})\simeq
\cat{K}^{\infty ,\emptyset}(\pj {R})
\]
\end{thm}

\begin{proof} 
According to Lemma \ref{st-h4}, Proposition \ref{st-stK} and Theorem \ref{st-st1}, 
we have the conclusion.
\end{proof}

\section{Cohen-Macaulay modules and the category of morphisms}\label{CMMorph}

Throughout this section let $R$ be an Iwanaga-Gorenstein ring.
In this section we introduce the category of morphisms between $R$-modules \cite{A1},
and study the relationship with the category of $\mrm{T}_2(R)$-modules.

Let us recall the following well-known result.

\begin{prop}\label{cm} 
The following hold.
\begin{enumerate}
\rmitem{1} $\CM R=\{Z^0(X)\ |\ X\in\cat{C} (\pj R), \mrm{H}^i (X) =0 ~(i \in \mathbb{Z} ) \}$.
\rmitem{2} $\underline{\CM} R\simeq\cat{K}^{\infty,\emptyset}(\pj R)$.
\end{enumerate}
\end{prop}

\begin{proof}
(1)
Let $M \in \CM R$, then
by Lemma \ref{st-h4} there is a complex $S \in \cat{K}^{+,\mrm{b}}(\pj R)$ such that
$S$ is quasi-isomorphic to $M$.
Let $C^{0}=\opn{Cok}(S^{-1} \to S^{0})$, then $C^{0}$ has a finite projective resolution and
$\Ext_{R}^i(C^{0}, R)=0$ for any $i >0$.
Therefore $C^{0}\in \pj R$, and we may assume $S:0 \to S^{0} \to S^{1} \to \cdots$.
Hence there is a complex $X \in \cat{C}(\pj R)$ such that 
$Z^0(X)\cong M$ and $\mrm{H}^i (X) =0 ~(i \in \mathbb{Z} )$.

(2)
It is easy to see that the functor $Z^0:\cat{K}^{\infty,\emptyset}(\pj R) \to \underline{\CM} R$
is a triangulated equivalence.
\end{proof}

\begin{defn}
Let $\mcal{B}$ an additive category.
We define the category $\Morph(\mcal{B})$ of morphisms in $\mcal{B}$ as follows.
\begin{itemize}
\item An object is a morphism $\alpha: X \to T$ in $\mcal{B}$. 
\item A morphism from $\alpha: X \to T$ to $\beta: X' \to T'$ is a pair $(f,g)$ of morphisms $f:X \to X'$ and $g:T \to T'$ such that
$g\alpha = \beta f$.
\end{itemize}
\end{defn}
When $\mcal{B}=\Mod R$ for a ring $R$, this category was studied systematically by Auslander \cite{A1} in the representation theory of algebras.
For an object $\alpha: X \to T$ in $\Morph(\Mod R)$, we have a $\opn{T}_2(R)$-module $X\oplus T$ where the multiplication is given by
\[(x,t){a\ b\choose 0\ c}:=(xa,\alpha xb+tc)\ \ \ (x\in X,\ t\in T,\ a,b,c\in R).\]
For a morphism $(f,g)$ from $\alpha: X \to T$ to $\beta: X' \to T'$, we have a morphism $f\oplus g : X\oplus T\to X'\oplus T'$ of $\opn{T}_2(R)$-modules
defined by
\[(f\oplus g)(x,t):=(fx,gt)\ \ \ (x\in X,\ t\in T).\]
We have the following fundamental observation \cite{A1}

\begin{prop}\label{T2}
For any ring $R$, we have an equivalence
\[\Morph(\Mod R)\to\Mod\opn{T}_2(R).\]
\end{prop}

We remark the following observation, where the former statement is well-known (see e.g. \cite{Iw}).

\begin{lem}\label{Lem:IwG}
Let $R$ be an Iwanaga-Gorenstein ring. Then so is $\opn{T}_2(R)$.
\end{lem}

Now we introduce full subcategories of $\Morph(\Mod R)$ corresponding to the full subcategory $\CM\opn{T}_2(R)$ of $\Mod\opn{T}_2(R)$.

\begin{defn}\label{morph_s}
Let $R$ be an Iwanaga-Gorenstein ring.
\begin{itemize}
\item[(1)] Define the full subcategory $\Morph^{\mrm{e}}(\CM R)$ of $\Morph(\Mod R)$ consisting of objects $\alpha: X \to T$ which is surjective and $X,T\in\CM R$ \footnote{Notice that $\opn{Ker}\alpha\in\CM R$ holds automatically in this case.}.
\item[(2)] Define the full subcategory $\Morph^{\mrm{m}}(\CM R)$ of $\Morph(\Mod R)$ consisting of 
objects $\alpha: X \to T$ which is injective and $X,T,\opn{Cok}\alpha\in\CM R$.
\end{itemize}
\end{defn}

We have the following result.
\begin{prop}\label{CMT2}
\footnote{This is independently obtained by Li and Zhang \cite{LZ}. See also \cite{C}.}
\begin{itemize}
\item[(1)] The equivalence in Proposition \ref{T2} induces an equivalence
\[\Morph^{\mrm{m}}(\CM R)\simeq\CM\opn{T}_2(R).\]
\item[(2)] $\alpha\mapsto\opn{cok}\alpha$ and $\alpha\mapsto\opn{ker}\alpha$ give mutually quasi-inverse equivalences
\[\xymatrix{\Morph^{\mrm{m}}(\CM R)\ar@<-2ex>[r]^{\opn{cok}}&\Morph^{\mrm{e}}(\CM R)\ar@<-1ex>[l]_{\opn{ker}}.}\]
\end{itemize}
\end{prop}

\begin{proof}
(1) Under the equivalence $\Mod\opn{T}_2(R)\to\Morph(\Mod R)$, any projective $\opn{T}_2(R)$-module corresponds to 
a split monomorphism $\alpha:X\to T$ of projective $R$-modules.

Assume that an object $\alpha:X\to T$ in $\Morph(\Mod R)$ corresponds to a Cohen-Macaulay $\opn{T}_2(R)$-module.
By Lemmas \ref{Lem:IwG} and \ref{cm}, there exists a commutative diagram
\[\begin{array}{ccccccc}
0\to&X&\to&X^0&\to&X^1&\to\cdots\\
&\downarrow^{\alpha}&&\downarrow^{\alpha^0}&&\downarrow^{\alpha^1}&\\
0\to&M&\to&M^0&\to&M^1&\to\cdots
\end{array}\]
of exact sequences such that $\alpha^i:X^i\to T^i$ corresponds to a projective $\opn{T}_2(R)$-module.
Since $\alpha^i$ is a split monomorphism of projective $R$-modules by the remark above,
we have that $\alpha$ is injective and $X,M\in\CM R$. Moreover we have $\opn{Cok}\alpha\in\CM R$
since we have an exact sequence
\[0\to\opn{Cok}\alpha\to\opn{Cok}\alpha^0\to\opn{Cok}\alpha^1\to\cdots\]
with projective $R$-modules $\opn{Cok}\alpha^i$.

Conversely let $\alpha:X\to T$ be an object in $\Morph^{\mrm{m}}(\CM R)$.
Since $X,\opn{Cok}\alpha\in\CM R$, there exist acyclic complexes $X_1$ and $X_3$, 
of which all cocycle are in $\CM R$, such that $\mrm{Z}^0(X_1)=X$ and $\mrm{Z}^0(X_3)=\opn{Cok}\alpha$.
Then we can construct a commutative diagram
\[\begin{array}{ccccccc}
&0&&0&&0&\\
&\downarrow&&\downarrow&&\downarrow&\\
0\to&X&\to&X^{0}_1&\to&X^{1}_1&\to\cdots\\
&\downarrow^\alpha&&\downarrow^{\alpha^0}&&\downarrow^{\alpha^1}&\\
0\to&T&\to&X^{0}_1\oplus X^{0}_3&\to&X^{1}_1\oplus X^{1}_3&\to\cdots\\
&\downarrow&&\downarrow&&\downarrow&\\
0\to&\opn{Cok}\alpha&\to&X^{0}_3&\to&X^{1}_3&\to\cdots\\
&\downarrow&&\downarrow&&\downarrow&\\
&0&&0&&0&\\
\end{array}\]
of exact sequences such that each vertical sequences are split exact.
Since each $\alpha^i$ corresponds to a projective $\opn{T}_2(R)$-module,
we have that $\alpha:X\to T$ in $\Morph(\Mod R)$ corresponds to a Cohen-Macaulay $\opn{T}_2(R)$-module.

(2) This is clear.
\end{proof}

We introduce subcategories of $\Morph^{\mrm{e}}(\CM R)$ to study $\CM\opn{T}_2(R)$.

\begin{defn} \label{SupCM}
Let $\underline{\Morph}^{\mrm{e}}(\CM R)$ be the stable category of $\Morph^{\mrm{e}}(\CM R)$.
We denote by $\underline{\Morph}^{10}(\CM R)$ (resp., $\underline{\Morph}^{11}(\CM R)$, $\underline{\Morph}^{01}(\CM R)$)
the full subcategory of $\underline{\Morph}^{\mrm{e}}(\CM R)$ consisting of objects
corresponding to the objects of the form $X \to 0$ (resp., $T \xarr{1} T$,
$P \to T$ with $P$ being projective) in $\Morph^{\mrm{e}}(\CM R)$.
\end{defn}

We have the following main result in this section.

\begin{thm} \label{st-st2}
Let $R$ be an Iwanaga-Gorenstein ring.
Then we have the following.
\begin{itemize}
\item[(1)] $(\underline{\Morph}^{01}(\CM R),\underline{\Morph}^{10}(\CM R),
\underline{\Morph}^{11}(\CM R))$ is a triangle of recollements in $\underline{\Morph}^{\mrm{e}}(\CM R)$.
\item[(2)] There is a recollement
\[\xymatrix{
\underline{\Morph}^{10}(\CM R)\ar@<-1ex>[r]^{i_{*}}
& \hspace{12pt}\underline{\Morph}^{\mrm{e}}(\CM R)\hspace{12pt}
\ar@<-2ex>[l]_{i^{*}} \ar@<2ex>[l]^{i^{!}} \ar@<-1ex>[r]^{j^{*}} 
& \hspace{6pt}\underline{\Morph}^{\mrm{e}}(\CM R)/\underline{\Morph}^{10}(\CM R)
\ar@<-2ex>[l]_{j_{!}} \ar@<2ex>[l]^{j_{*}}
}\]
such that the essential image $\opn{Im}j_{!}$ (resp., $\opn{Im}j_{*}$)
is $\underline{\Morph}^{01}(\CM R)$ (resp., $\underline{\Morph}^{11}(\CM R)$).
\item[(3)] We have equivalences 
\[\underline{\CM} R \simeq \underline{\Morph}^{10}(\CM R)
\simeq \underline{\Morph}^{11}(\CM R)
\simeq \underline{\Morph}^{01}(\CM R).\]
\end{itemize}
\end{thm} 

\begin{proof}
It is clear that $\Hom_{\underline{\Morph}^{\mrm{e}}(\CM R)}
(\underline{\Morph}^{01}(\CM R),\underline{\Morph}^{10}(\CM R))=0$,\\
$\Hom_{\underline{\Morph}^{\mrm{e}}(\CM R)}
(\underline{\Morph}^{10}(\CM R),\underline{\Morph}^{11}(\CM R))=0$ and\\
$\Hom_{\underline{\Morph}^{\mrm{e}}(\CM R)}
(\underline{\Morph}^{11}(\CM R),\underline{\Morph}^{01}(\CM R))=0$.
Given any object $(\alpha:X \to T) \in\Morph^{\mrm{e}}(\CM R)$, 
we have an exact sequence in $\Morph^{\mrm{e}}(\CM R)$:
\[\begin{CD}
0 @>>> \opn{Ker}\alpha @>>> X @>\alpha>> T @>>> 0 \\
@. @VV0_{\opn{Ker}\alpha} V @VV\alpha V @VV1_T V  \\
0 @>>> 0 @>>> T @>>> T @>>> 0
\end{CD}\]
Therefore we have the triangle 
$0_{\opn{Ker}\alpha} \to \alpha \to 1_T \to \Sigma 0_{\opn{Ker}\alpha}$
in $\underline{\Morph}^{\mrm{e}}(\CM R)$.
Let $0 \to X \to P \to \Sigma X \to 0$ and 
$0 \to \Sigma^{-1}X \to Q \to X \to 0$ 
be exact sequences in $\fmod R$ such that
$P, Q$ are finitely generated projective and $\Sigma^{-1}X, \Sigma X \in \CM R$.
Then we have morphisms of exact sequences:
\[\begin{CD}
0 @>>> \Sigma^{-1}X @>>> Q @>>>  X @>>> 0 \\
@. @VV0_{\Sigma^{-1} X}  V @VV\sigma V @VV\alpha V  \\
0 @>>> 0 @>>> T @>>>  T @>>> 0
\end{CD}\]
\[\begin{CD}
0 @>>> X @>>> P @>>> \Sigma X @>>> 0 \\
@. @VV\alpha V @VV\tau V @VV1_{\Sigma X} V  \\
0 @>>> T @>>> M @>>> \Sigma X @>>> 0
\end{CD}\]
Since $0_{\Sigma^{-1} X}\simeq  \Sigma^{-1} 0_{X}$ and $1_{\Sigma X}\simeq  \Sigma 1_{X}$ 
in $\underline{\Morph}^{\mrm{e}}(\CM R)$,
we have the triangles
$\sigma \to \alpha \to 0_{X} \to \Sigma\sigma $ and
$1_{X} \to \alpha \to \tau \to \Sigma 1_{X}$ in $\underline{\Morph}^{\mrm{e}}(\CM R)$.
It is easy to see that $\underline{\Morph}^{10}(\CM R)$ is equivalent to $\underline{\CM} R$,
then we have the statement by Propositions \ref{added} and \ref{p20Apr30}.
\end{proof}

\section{Triangle equivalence between recollements}\label{trieqrecoll}

We prove a triangle equivalence between $\underline{\CM}\opn{T}_2(R)$ and 
$(\cat{K}^{\infty,\mrm{b}}/\cat{K}^{\mrm{b}})(\pj {R})$. 
The rough sketch of the proof is the following. 
By Proposition \ref{CMT2}, we show there is a triangle equivalence from 
$\underline{\Morph}^{\mrm{e}}(\CM R)$ to $(\cat{K}^{\infty, \emptyset}/\cat{K}^{\mrm{b}})(\pj {R})$. 
We begin with building a functor 
$F: \Morph^{\mrm{e}}(\CM R) \to \cat{K}^{\infty, \emptyset}(\pj {R})$. 
Next we see that $F$ induces a triangle functor 
$\underline{F}: \underline{\Morph}^{\mrm{e}}(\CM R) \to (\cat{K}^{\infty, \emptyset}/\cat{K}^{\mrm{b}})(\pj {R})$. 
The point is that both domain and target categories have triangles of recollements and 
$\underline{F}$ sends one triangle of recollements to the other. 
So if $\underline{F}$ restricts to an equivalence between 
some composing subcategories of these triangles of recollements, 
then $\underline{F}$ itself is an equivalence from Proposition \ref{p20Apr30}.

Let $\alpha : X_\alpha \to T_\alpha$ be an object of $\Morph^{\mrm{e}}(\CM R)$ and let 
$F_{X_\alpha}$ and $F_{T_\alpha}$ be acyclic complexes of finitely generated projective $R$-modules 
such that ${\rm H}^0 ( \tau _{\le 0} F_{X_\alpha} ) = {X_\alpha}$ and 
${\rm H}^0 ( \tau _{\le 0} F_{T_\alpha} )= {T_\alpha}$. 
Make a complex $F_\alpha$ as 
\[ \tau _{\le 0} F_{\alpha} = \tau _{\le 0} F_{X_\alpha}, ~
\tau _{\ge 1} F_{\alpha} = \tau _{\ge 1} F_{T_\alpha}, ~
d_{F_\alpha} = \epsilon _{T_\alpha} \alpha \rho _{X_\alpha}  \]  
where $\rho _{X_\alpha} :F_{X_\alpha}^0 \to {X_\alpha}$ ($\epsilon _{T_\alpha} :{T_\alpha} \to F_{T_\alpha}^1$ ) 
is a natural surjective (resp. injective) map. 

\begin{lem}\label{K8Jul55}
1)~For a morphism $f\in \Hom_{\Morph^{\mrm{e}}(\CM R)}(\alpha , \beta )$, there is a chain map 
$F_f : F_\alpha \to F_\beta $ such that 
$({\rm H}^0 ( \tau _{\le 0} F_{f} ), {\rm H}^1 ( \tau _{\ge 1} F_{f} ) ) =f$. \\
2)~For morphisms $f \in\Hom_{\Morph^{\mrm{e}}(\CM R)}(\alpha , \beta )$ and $g \in \Hom_{\Morph^{\mrm{e}}(\CM R)}(\beta , \gamma )$, 
we may choose chain maps $F_f$, $F_g$ and $F_{gf}$ as $F_{gf} =F_g F_f$. \\
3)~An exact sequence 
$0\to \alpha \stackrel{f}{\to} \beta \stackrel{g}{\to} \gamma \to 0$ in $\Morph^{\mrm{e}}(\CM R)$ induces an exact sequence 
$0\to F_\alpha \stackrel{F_f}{\to}  F_\beta \stackrel{F_g}{\to} F_\gamma \to 0$ 
by the choice of complexes and chain maps. \\

\end{lem}

\begin{proof}
1)~Since projective $R$-modules are projective-injective objects in 
$\CM R$, linear maps $f_X : X_\alpha \to X_\beta$ and  
${f_T} : {T_\alpha} \to {T_\beta}$ induce 
chain maps between complexes 
$ {f_X}^\bullet : \tau _{\le 0} F_\alpha \to \tau _{\le 0} F_\beta$ and 
$ {f_T}^{\bullet} : {\tau _{\ge 1} F_\alpha } \to {\tau _{\ge 1} F_\beta} $ 
respectively. Putting $F_f : F_\alpha \to F_\beta$ as 
\[ F_f^i = \begin{cases}
f_X^i & (i\le 0)\cr
{f_T^{i} } &(i\ge 1 )\cr
\end{cases} , \] 
we see $F_f$ is a chain map because 
$d_{F_\beta}^0  f_X ^0 = (\epsilon_{T_\beta} \beta \rho_{X_\beta} )  f_X ^0 
=\epsilon_{T_\beta} \beta (f_X \rho_{X_\alpha} )   
=\epsilon_{T_\beta} (f_T \alpha ) \rho_{X_\alpha}    
=({f_T^{1}} \epsilon_{T_\alpha} ) \alpha  \rho_{X_\alpha}
={f_T^{1} } d_{F_\alpha}^0 $. 

2)~It is obvious from the construction of $F_f$ and $F_g$ as in 1). 

3)~The given exact sequence implies a commutative diagram with exact rows in 
$\CM R$: 
\[ \begin{matrix}0&\to&X_ \alpha&\stackrel{f_X}{\to}&X_\beta&\stackrel{g_X}{\to}&X_\gamma&\to&0\cr 
&&\mapdown{\alpha}&&\mapdown{\beta}&&\mapdown{\gamma}&&\cr
0&\to&T_ \alpha&\stackrel{f_T}{\to}&T_\beta&\stackrel{g_T}{\to}&T_\gamma&\to&0\cr \end{matrix} \] 
The upper row induces an exact sequence of projective resolutions 
\[ \begin{matrix}
0&\to&\tau _{\le 0} F_ \alpha&{\to}&\tau _{\le 0} F_\beta&{\to}&\tau _{\le 0} F_\gamma&\to&0, \cr 
\end{matrix} \]
and the lower row induces an exact sequence of projective coresolutions
as well: 
\[ \begin{matrix}
0&\to&{\tau _{\ge 1} F_ \alpha }&{\to}&{\tau _{\ge 1} F_ \beta }&{\to}&
{\tau _{\ge 1} F_ \gamma }&\to&0\cr 
\end{matrix} \] 
Thus we obtain a commutative diagram with exact rows in 
$\cat{C}(\fmod R)$
\[ \begin{matrix}
0&\to&\tau _{\le 0} F_ \alpha&{\to}&\tau _{\le 0} F_\beta&{\to}&\tau _{\le 0} F_\gamma&\to&0\cr 
&&\mapdown{\rho _{X_\alpha}}&&\mapdown{\rho _{X_\beta}}&&\mapdown{\rho _{X_\gamma}}&&\cr
0&\to&X_ \alpha&\stackrel{f_X}{\to}&X_\beta&\stackrel{g_X}{\to}&X_\gamma&\to&0\cr 
&&\mapdown{\alpha}&&\mapdown{\beta}&&\mapdown{\gamma}&&\cr
0&\to&T_ \alpha&\stackrel{f_T}{\to}&T_\beta&\stackrel{g_T}{\to}&T_\gamma&\to&0\cr 
&&\mapdown{\epsilon _{T_\alpha}}&&\mapdown{\epsilon _{T_\beta}}&&\mapdown{\epsilon _{T_\gamma}}&&\cr
0&\to&\tau _{\ge 1} F_ \alpha&{\to}&\tau _{\ge 1} F_\beta&{\to}&\tau _{\ge 1} F_\gamma&\to&0\cr 
\end{matrix} \]  
whose composite is the desired sequence. 

\end {proof}

\begin{lem}\label{F}
The operation $F$ gives a full functor 
$\Morph^{\mrm{e}}(\CM R) \to {\sf K}^{\infty,\mrm{b}}(\pj {R})$.
\end{lem}

\begin{proof}
We shall see first that $F_\alpha$ is uniquely determined 
as an object of ${\sf K}^{\infty ,\mrm{b}}(\pj {R})$ by $\alpha$, 
independent of the choice of $F_{X_\alpha}$ and $F_{T_\alpha}$.
Suppose other acyclic complexes of projective modules 
$F'_{X_\alpha}$ and $F' _{T_\alpha}$ also 
have the property  
${\rm H}^0 (\tau _{\le 0} F'_{X_\alpha} )= {X_\alpha}$ and 
${\rm H}^0 (\tau _{\le 0} F'_{T_\alpha} )= {T_\alpha}$. 
Both $\tau _{\le 0} F_{X_\alpha} $ and $\tau _{\le 0} F'_{X_\alpha} $ are 
projective resolutions of ${X_\alpha}$, hence are isomorphic in ${\sf K}^{-, \mrm{b} }(\pj {R})$. 
Similarly $\tau _{\ge 1} F_{T_\alpha} $ and $\tau _{\ge 1} F'_{T_\alpha} $ are 
isomorphic in ${\sf K}^{+, \mrm{b}}(\pj {R})$.
Therefore, in ${\sf K}^{\infty ,\mrm{b}}(\pj {R})$, the complex $F_\alpha$ is isomorphic to 
the complex $F' _\alpha$ obtained from  $F'_{X_\alpha}$ and $F' _{T_\alpha}$ as 
\[ \tau _{\le 0} F'_{\alpha} = \tau _{\le 0} F'_{X_\alpha}, ~
\tau _{\ge 1} F'_{\alpha} = \tau _{\ge 1} F'_{T_\alpha}, ~
d_{F'_\alpha} = {\epsilon '}_{T_\alpha} \alpha {\rho '}_{X_\alpha}  \]  
where 
${\rho '}_{X_\alpha} : {F'}^0 _{X_\alpha} \to X_\alpha$ is the natural projection and 
${\epsilon '}_{T_\alpha} : T_\alpha \to F^{\prime 1}_{T_\alpha}$ is the natural inclusion.  

Next we see that $f \in \Hom _{\Morph^{\mrm{e}}(\CM R)}(\alpha , \beta )$ determines 
$F_f \in \Hom _{{\sf K}^{\infty ,\mrm{b}}(\pj {R})}(F_ \alpha , \\
F_\beta )$. 
Suppose two chain maps $\varphi _1$ and $\varphi _2$ both satisfy 
${\rm H}^0 (\tau _{\le 0} \varphi _i )= f_X$, 
${\rm H}^1 (\tau _{\ge 1} \varphi _i )= f_T$ \quad (i =1,2). 
Then $\tau _{\le 0} \varphi _1$ and $\tau _{\le 0} \varphi _2$ 
($\tau _{\ge 1} \varphi _1$ and $\tau _{\ge 1} \varphi _2$) are homotopic , hence 
so are $\varphi _1$ and $\varphi _2$. 

We have already seen that $F$ is functorial from Lemma~\ref{K8Jul55} 1). 

We shall show that $F$ is full. Let $\alpha$ and $\beta$ be objects of $\underline{\Morph}^{\mrm{e}}(\CM R)$ and 
let $\varphi$ be a morphism from $F_\alpha$ to $F_\beta$. 
Then $f_X = {\rm H}^0 (\tau _{\le 0} \varphi )$ and 
$f_T = {\rm H}^1 (\tau _{\ge 1} \varphi )$ satisfy 
$\beta f_X = f_T \alpha$ because 
$\epsilon _{T_\beta} (\beta f _X ) \rho _{X_\alpha} = 
\epsilon _{T_\beta} (f _T \alpha ) \rho _{X_\alpha}$, 
$\rho _{X_\alpha}$ is surjective and $\epsilon _{T_\beta}$ is injective. 
Hence $f = (f _X , f _T )$ is a morphism from 
$\alpha$ to $\beta$ in $\Morph^{\mrm{e}}(\CM R)$. 
Obviously $\varphi = F_f$. 
\end{proof}

\begin{lem}\label{Fbar} 
The functor $F$ induces a functor 
\[
\underline{F} : \underline{\Morph}^{\mrm{e}}(\CM R) \to ({\sf K}^{\infty,\mrm{b}}/ {\sf K}^{\mrm{b}})(\pj {R}). 
\]
\end{lem}

\begin{proof} Let $p$ be an object of $\Morph^{\mrm{e}}(\CM R)$. 
If $F_p$ is bounded, then both $X_p$ and $T_p$ have finite projective dimension hence 
are projective modules. If $p$ is a projective object in $\Morph^{\mrm{e}}(\CM R)$, 
then $X_p$ and $T_p$ are projective modules hence
$F_p$ is just given by $F_p ^0 =X_p$, $F_p ^1 =T_p$, $d_{F_p} ^0 =p$ and 
$F_p ^i =0 \quad (i\neq 0,1 )$. 
\end{proof} 

For an object $A$ and a morphism $\varphi$ in ${\sf K}^{\infty,\mrm{b}}(\pj {R})$, 
$\underline{A}$ and $\underline{\varphi}$ denote $Q(A)$ and $Q(\varphi )$ respectively 
where $Q$ is the canonical quotient functor 
$Q: {\sf K}^{\infty,\mrm{b}}(\pj {R}) \to ({\sf K}^{\infty,\mrm{b}}/ {\sf K}^{\mrm{b}})(\pj {R})$.

\begin{prop}\label{K665}
The functor $\underline{F} : \underline{\Morph}^{\mrm{e}}(\CM R)\to ({\sf K}^{\infty,\mrm{b}}/ {\sf K}^{\mrm{b}})(\pj {R})$
is a triangle functor. 
\end{prop}

\begin{proof}
First we show the existence of a functorial isomorphism 
$\underline{i_\alpha} : \underline{F_{\Sigma \alpha}} \to \underline{\Sigma F_\alpha}$ 
for each object $\alpha$ of $\Morph^{\mrm{e}}(\CM R)$. 
Take an exact sequence 
\[ 0 \to \alpha \stackrel{\epsilon _\alpha}{\to} q_\alpha \stackrel{\pi_\alpha}{\to} \Sigma \alpha \to 0 \] 
in the Frobenius category $\Morph^{\mrm{e}}(\CM R)$
such that $q_\alpha$ is an injective object.
Then it induces an exact sequence 
\[ 0 \to F_\alpha \stackrel{F_{\epsilon _\alpha}}{\to} F_{q_\alpha} \stackrel{F_{\pi_\alpha}}{\to}  F_{\Sigma \alpha} \to 0 \] 
in 
$\cat{C}(\fmod R)$.
Let $I_{F_\alpha}$ be the mapping cone of the identity map
$1:F_\alpha \to F_\alpha$.
Then 
we have a commutative diagram with exact rows in 
$\cat{C}(\fmod R)$
\[  \begin{matrix}0&\to&F_ \alpha& \stackrel{F _{\epsilon_\alpha}}{\to}&F_{q_\alpha} 
&\stackrel{F_ {\pi_{\alpha}}}{\to}&F_{\Sigma \alpha}&\to&0\cr 
&&\|&&\mapdown{j_\alpha}&&\mapdown{i_\alpha}&&\cr
0&\to&F_ \alpha&\to&I_{F_\alpha}&{\to}&\Sigma F_{\alpha}&\to&0\cr\end{matrix} \]
which implies a triangle in ${\sf K}^{\infty , \mrm{b}} (\pj R)$ 
\[ F_{\alpha} \xarr{F_{\epsilon _\alpha}} F_{q_\alpha} \xarr{F_{\pi_{\alpha}}} F_{\Sigma \alpha} 
\xarr{i_\alpha} \Sigma F_{\alpha}. 
\] 
Since $F_{q_\alpha}$ belongs to ${\sf K}^{\mrm{b}}(\pj R)$, 
$\underline{i_\alpha}$ is an isomorphism in 
$({\sf K}^{\infty , \mrm{b}}/ {\sf K}^{\mrm{b}})(\pj R)$. 
Moreover, $\underline{i_\alpha}$ is functorial. 
Let $f: \alpha \to \beta$ be a morphism in $\Morph^{\mrm{e}}(\CM R)$. 
We have a commutative diagram with exact rows
\[ \begin{matrix}0&\to&\alpha&\stackrel{\epsilon_\alpha}{\to}&q_\alpha &\stackrel{\pi _{\alpha}}{\to}&\Sigma \alpha&\to&0\cr 
&&\mapdown{f}&&\mapdown{f_q}&&\mapdown{\Sigma f}&&\cr
0&\to&\beta&\stackrel{\epsilon _\beta}{\to}&q_\beta & \stackrel{\pi _{\beta}}{\to}&\Sigma \beta&\to&0.\cr\end{matrix} \] 
where $q_\alpha$ and $q_\beta$ are injective objects. 
Since $F_{\epsilon _\beta} F_f =F_{f_q} F_{\epsilon _\alpha}$, $F_{\pi _\beta} F_{f_q} = F_{\Sigma f} F_{\pi _\beta}$ in 
${\sf K}^{\infty , \mrm{b}}(\pj R)$, there is a morphism of triangles 
\[\begin{CD}
F_{\alpha} @>F_{\epsilon_\alpha}>> F_{q_\alpha} @>F_{\pi _\alpha}>> F_{\Sigma\alpha} @>i_{\alpha}>> \Sigma F_{\alpha} \\
@VV F_{f}V @VV F_{f_q}V @VV F_{\Sigma f}V @VV \Sigma F_{f}V\\
F_{\beta} @>F_{\epsilon_\beta}>> F_{q_\beta} @>F_{\pi _\beta}>> F_{\Sigma\beta} @>i_\beta>>\Sigma F_{\beta} \\
\end{CD}\] 
Particularly, $\Sigma F_f i_\alpha = i_\beta F_{\Sigma f}$. 

Next, let 
\[ \underline{\alpha} \stackrel{\underline{f}}{\to} \underline{\beta} \stackrel{\underline{g}}{\to} 
\underline{\gamma}  \stackrel{\underline{h}}{\to}
\Sigma \underline{\alpha} \] be a triangle in $\underline{\Morph}^{\mrm{e}}(\CM R)$. 
That is, 
we have
an injective object $q_\alpha$ such that there is
a commutative diagram  in $\Morph^{\mrm{e}}(\CM R)$
with exact rows:
\[ \begin{matrix}0&\to&\alpha&\stackrel{\epsilon_\alpha}{\to}&q_\alpha&\stackrel{\pi _\alpha}{\to}&\Sigma \alpha&\to&0\cr 
&&\mapdown{f}&&\mapdown{w}&&\|&&\cr
0&\to&\beta&\stackrel{g}{\to}&\gamma&\stackrel{h}{\to}&\Sigma \alpha&\to&0.\cr\end{matrix} \] 
The induced exact sequence 
\[ \begin{matrix}0&\to&F_ \alpha& \stackrel{F _{\epsilon_\alpha}}{\to}&F_{q_\alpha}&\stackrel{F_ {\pi _\alpha}}{\to}&F_{\Sigma \alpha}&\to&0\cr 
\end{matrix} \]
is completed as an diagram with exact rows in 
$\cat{C}(\fmod R)$:
\begin{equation}\label{16Octp50}
 \begin{matrix}0&\to&F_ \alpha& \stackrel{F _{\epsilon_\alpha}}{\to}&F_q&\stackrel{F_ {\pi _\alpha}}{\to}&F_{\Sigma \alpha}&\to&0\cr 
&&\mapdown{F_f}&&\mapdown{s}&&\|&&\cr
0&\to&F_ \beta&\stackrel{c}{\to}&E&\stackrel{n}{\to}&F_{\Sigma \alpha}&\to&0\cr\end{matrix} \end{equation}
The bottom row induces a commutative diagram in $\CM R$ with exact rows 
\[ \begin{matrix}0&\to&X_ \beta&\stackrel{g'_X}{\to}&X_E&\stackrel{h' _X}{\to}&X_{\Sigma \alpha}&\to&0\cr 
&&\mapdown{\beta}&&\mapdown{\lambda}&&\mapdown{\Sigma \alpha}&&\cr
0&\to&T_ \beta&\stackrel{g' _T}{\to}&T_E&\stackrel{h' _T}{\to}&T_{\Sigma \alpha}&\to&0\cr \end{matrix} \] 
where $\lambda$ is the canonical map between $X_E = \opn{Cok} d_E^{-1}$ and $T_E = \opn{Ker} d_E ^{1}$. 
Hence $\lambda$ belongs to $\Morph^{\mrm{e}}(\CM R)$ and $E= F_\lambda$ since 
$\rm{H}^i (E) = 0$ for $i\neq 0$. 
We have $c =F_{g'}$ and $n=F_{h'}$ with $g'= ({g'}_X , {g'}_T )$ and $h'= ({h'}_X , {h'}_T )$. 
By the same argument in Lemma~\ref{F}, there exists $w' \in \Hom _{\CM} (q_\alpha , \lambda )$ such that $s=F_{w'}$. 
Thus we have a diagram with exact rows: 
\[ \begin{matrix}0&\to&\alpha&\stackrel{\epsilon_\alpha}{\to}&q_\alpha&\stackrel{\pi _\alpha}{\to}&\Sigma \alpha&\to&0\cr 
&&\mapdown{f}&&\mapdown{w'}&&\|&&\cr
0&\to&\beta&\stackrel{g'}{\to}&\lambda&\stackrel{h'}{\to}&\Sigma \alpha&\to&0.\cr\end{matrix} \] 
which shows $\gamma \cong \lambda$, $g' =g$ and $h' =h$ via this isomorphism. 
Hence $E= F_\gamma$ and (\ref{16Octp50}) implies a morphism between triangles in $\cat{K}^{\infty,\mrm{b}}(\pj {R})$:

\[\begin{CD}
F_{\alpha} @>F_{\epsilon_\alpha}>> F_{q_\alpha} @>F_{\pi_\alpha}>> F_{\Sigma\alpha} @>i_{\alpha}>>\Sigma F_{\alpha} \\
@VV F_{f}V @VV F_{w'}V @| @VV \Sigma F_{f}V\\
F_{\beta} @>F_g>> F_{\gamma} @>F_h>> F_{\Sigma\alpha} @>\Sigma {F_f} {i_\alpha}>>\Sigma F_{\beta} \\
\end{CD}\]

In $({\sf K}^{\infty , \mrm{b}}/ {\sf K}^{\mrm{b}})(\pj R)$, $\underline{i_\alpha}$ is an isomorphism 
hence we get a triangle 
\[ \underline{F_{\beta}} \xarr{\underline{F_g}} \underline{F_{\gamma}} \xarr{\underline{F_h}} \underline{\Sigma F_{\alpha} }
\xarr{\underline{\Sigma F_{f}}} \underline{\Sigma F_{\beta}} \] which implies a triangle 
\[ \underline{F_{\alpha}} \xarr{\underline{F_f}} \underline{F_{\beta}} \xarr{\underline{F_g}} \underline{F_{\gamma} } 
\xarr{\underline{F_{h}}}  \underline{\Sigma F_{\alpha}}. \]
\end{proof}

\begin{lem} \label{st-eq2bis}
The functor 
$\underline{F}\vert_{\underline{\Morph}^{11}(\CM R)}:\underline{\Morph}^{11}(\CM R) \to 
\cat{K}^{\infty,\emptyset}(\pj {R})$
is a triangle equivalence.
\end{lem} 

\begin{proof}
Define the functor $\underline{H}:\underline{\CM} R \to \underline{\Morph}^{11}(\CM R)$ by
$H(X)=(X\xarr{1_{X}} X)$ for any object $X \in \CM R$.
Then $H$ is an triangle equivalence.
Moreover, it is easy to see that the functor$\underline{F}\vert_{\underline{\Morph}^{11}(\CM R)}\circ\underline{H}:
\underline{\CM R} \to \cat{K}^{\infty,\emptyset}(\pj {R})$ is also an equivalence. 
Therefore we have the statement.
\end{proof}

\begin{lem} \label{st-eq1}
Let $R$ be an Iwanaga-Gorenstein ring. 
Then the functor $\underline{F}: \\
\underline{\Morph}^{\mrm{e}}(\CM R) \to \cat{K}^{\infty,\mrm{b}}(\pj {R})$ 
sends a triangle $(\underline{\Morph}^{01}(\CM R),\underline{\Morph}^{10}(\CM R), \\
\underline{\Morph}^{11}(\CM R))$ of recollements to a triangle
$((\cat{K}^{+,\mrm{b}}/\cat{K}^{\mrm{b}})(\pj {R}),
(\cat{K}^{-,\mrm{b}}/\cat{K}^{\mrm{b}})(\pj {R}), \\
\cat{K}^{\infty ,\emptyset}(\pj {R}))$ of recollements. 
\end{lem}

\begin{thm} \label{mth}
Let $R$ be an Iwanaga-Gorenstein ring. There is a triangle equivalence from $\underline{\Morph}^{\mrm{e}}(\CM R)$
to $(\cat{K}^{\infty,\mrm{b}}/\cat{K}^{\mrm{b}})(\pj {R})$. 
\end{thm}

\begin{proof}
>From Theorem~\ref{st-h3} and Theorem~\ref{st-st2}, there are triangles of recollements 
$(\underline{\Morph}^{01}(\CM R), \underline{\Morph}^{10}(\CM R), \underline{\Morph}^{11}(\CM R))$
in $\underline{\Morph}^{\mrm{e}}(\CM R)$ \\
and 
$((\cat{K}^{+,\mrm{b}}/\cat{K}^{\mrm{b}})(\pj {R}),
(\cat{K}^{-,\mrm{b}}/\cat{K}^{\mrm{b}})(\pj {R}), 
\cat{K}^{\infty ,\emptyset}(\pj {R}))$ in $(\cat{K}^{\infty,\mrm{b}}/ \cat{K}^{\mrm{b}})(\pj {R})$. 
By Lemma \ref{st-eq1}, the functor 
$\underline{F}:\underline{\Morph}^{\mrm{e}}(\CM R)\to (\cat{K}^{\infty ,\mrm{b}}/ \cat{K}^{\mrm{b}})(\pj {R})$ restricts to 
triangle functors 
$\underline{F}\vert_{\underline{\Morph}^{01}(\CM R)}:\underline{\Morph}^{01}(\CM R)\to (\cat{K}^{+,\mrm{b}}/ \cat{K}^{\mrm{b}})(\pj {R})$, 
$\underline{F}\vert_{\underline{\Morph}^{10}(\CM R)}:\underline{\Morph}^{10}(\CM R)\to (\cat{K}^{-,\mrm{b}}/ \cat{K}^{\mrm{b}})(\pj {R})$  and 
$\underline{F}\vert_{\underline{\Morph}^{11}(\CM R)}:\underline{\Morph}^{11}(\CM R)\to \\
\cat{K}^{\infty ,\emptyset }(\pj {R})$.   
And $\underline{F}\vert_{\underline{\Morph}^{11}(\CM R)}$ is an equivalence by Lemma \ref{st-eq2bis}. 
Then Proposition \ref{p20Apr30} gives us the conclusion.
\end{proof}

The quasi-inverse of $\underline{F}$ is constructed in the following way. 
There is a recollement  
\[\xymatrix{
\cat{K}^{\infty,\emptyset}(\pj {R}) \ar@<-1ex>[r]^{k_{*}}
& \cat{K}^{\infty,\mrm{b}}(\pj {R})
\ar@<-2ex>[l]_{k^{*}} \ar@<2ex>[l]^{k^{!}} \ar@<-1ex>[r]^{l^{*}} 
& \quad (\cat{K}^{\infty,\mrm{b}}/\cat{K}^{\infty,\emptyset})(\pj {R}) 
\ar@<-2ex>[l]_{l_{!}} \ar@<2ex>[l]^{l_{*}}
 }\] 
With respect to two functors $X={k_*}{k^!}$ and $T={k_*}{k^*}$, 
the adjunction arrows
$k_{*}k^{!} \to \bsym{1}_{\cat{K}^{\infty,\mrm{b}}(\pj {R})}$
and $\bsym{1}_{\cat{K}^{\infty,\mrm{b}}(\pj {R})} \to k_{*}k^{*}$ 
imply morphisms 
$\xi_L : X_L \to L$ and $\zeta _L : L \to T_L$ for each object $L$ of 
$\cat{K}^{\infty,\mrm{b}}$. Put ${\lambda _L} ={\zeta _L}{\xi_L}$. 
We may assume ${\lambda _L}^0$ is surjective hence get a surjective map 
${Z^1} \lambda _L : {Z^1}{X_L} \to {Z^1}{T_L}$ where ${Z^1}{X_L}= \opn{Ker}d^1_{X_L}$ and 
${Z^1}{T_L}= \opn{Ker}d^1_{T_L}$. 
Let $f/s : L \to M$ be a morphism of $(\cat{K}^{\infty,\mrm{b}}/ \cat{K}^{\mrm{b}})(\pj {R})$ 
given by a diagram of morphisms 
$L \stackrel{s}{\leftarrow} L' \stackrel{f}{\rightarrow} M$ in ${\cat{K}^{\infty,\mrm{b}}(\pj {R})}$ 
with $C(s) \in {\cat{K}^{\mrm{b}}(\pj {R})}$.  
Since $Xs$ and $Ts$ are isomorphisms, $(({Z^1}{Xf}) ({Z^1}{Xs})^{-1}, ({Z^1}{Tf}) ({Z^1}{Ts})^{-1} )$ belongs to 
$\Hom _{\underline{\Morph}^{\mrm{e}}(\CM R)} ( {Z^1 \lambda _L}, {Z^1 \lambda _M} )$.  

\begin{prop}\label{quasi-inverse}
The operation ${Z^1 \lambda }$ induces a functor 
$\underline{Z^1 \lambda } :  (\cat{K}^{\infty,\mrm{b}}/ \cat{K}^{\mrm{b}})(\pj {R}) \\
\to \underline{\Morph}^{\mrm{e}}(\CM R)$ 
which is the quasi-inverse functor of $\underline{F}$. 
\end{prop}

\begin{proof}
Let $\underline{G}$ be the quasi-inverse of $\underline{F}$. 
Obviously $\underline{F _{Z^1 \lambda _L}}$ is isomorphic to $L$ in $(\cat{K}^{\infty,\mrm{b}}/ \cat{K}^{\mrm{b}})(\pj {R})$ 
for every object $L$. Hence we have $\underline{G}(L) = \underline{G}(F_{Z^1 \lambda _L}) = 
\underline{G}\underline{F}(Z^1 {\lambda _L }) =Z^1 {\lambda _L } $. 
Similarly we have $\underline{G}(f/s) = (({Z^1}{Xf}) ({Z^1 Xs})^{-1}, \\
({Z^1}{Tf}) ({Z^1}{Ts})^{-1} )$ for 
a morphism $L \stackrel{s}{\leftarrow} L' \stackrel{f}{\rightarrow} M$  in 
$\Hom _{(\cat{K}^{\infty,\mrm{b}}/ \cat{K}^{\mrm{b}})(\pj {R}) } (L,M)$. 
\end{proof}

\section{Homotopy categories of infinitely generated modules} \label{secLCM} 

In the case that a Iwanaga-Gorenstein ring $R$ has a two-sided injective resolution,
we study homotopy categories of infinitely generated projective modules and 
give the triangle equivalence analogous to one in Theorem \ref{mth}. 
Also in this infinite case, the key is a triangle of recollements but 
is obtained by a quite different way from in the finite case. 
Namely we use the equivalence between $\cat{K} (\Pj {R})$ and $\cat{K} (\Ij {R})$ \cite{IK}, 
which gives us not only a projective version but also injective version of the result. 
Recall that we say that $R$ has a two-sided injective resolution if
there is an $R$-bimodule complex $V$ which is
an injective resolution of $R$ as right $R$-modules and as left $R$-modules.

The main result of this section is the following. 

\begin{thm}\label{Bmth}
Let $R$ be an Iwanaga-Gorenstein ring with a finite two-sided injective resolution. 
Then there is a triangle equivalence from $\underline{\LCM}\opn{T}_2(R)$
to \\
$(\cat{K}^{\infty,\mrm{b}}/\cat{K}^{\mrm{b}})(\Pj {R})$. 
\end{thm} 
 
\begin{lem}\label{Lem:2Ijres}
Let $R$ be an Iwanaga-Gorenstein ring.
If $R$ has a two-sided injective resolution, then so does $\opn{T}_2(R)$.
\end{lem}

\begin{proof}
We only prove the latter statement.
If $0 \to R \to V^{0} \xarr{d^0} V^{1} \xarr{d^1} \cdots V^{n} \to 0$ is a two-sided injective resolution,  then
{\small
\[
0 \to
\opn{T}_{2}(R) \to
\left(\begin{smallmatrix}V^{0} & V^{0} \\V^{0} & V^{0}\end{smallmatrix}\right)\  
\xarr{\Delta^{0}}
\left(\begin{smallmatrix} V^{1} & V^{1} \\V^{0}\oplus V^{1} & V^{1}\end{smallmatrix}\right)\  \xarr{\Delta^{1}}
\cdots \to 
\left(\begin{smallmatrix} V^{n} & V^{n} \\V^{n-1}\oplus V^{n} & V^{n}\end{smallmatrix}\right)\  \to
\left(\begin{smallmatrix}0 & 0 \\V^{n} & 0\end{smallmatrix}\right)\  \to 0
\]
}
is a two-sided injective resolution, where any
$
\left(\begin{smallmatrix} V^{i} & V^{i} \\
V^{i-1}\oplus V^{i} & V^{i}\end{smallmatrix}\right)\
$
is an ${\opn{T}_2(R)}$-bimodule by the following actions
\[\begin{aligned}
\left(\begin{smallmatrix}v_{11}^{i} & v_{12}^{i} \\
(u_{21}^{i-1},v_{21}^{i}) & v_{22}^{i}\end{smallmatrix}\right)\
\left(\begin{smallmatrix}a & b \\
0 & c\end{smallmatrix}\right)
=
\left(\begin{smallmatrix}v_{11}^{i}a & v_{11}^{i}b+v_{12}^{i}c \\
(u_{21}^{i-1}a,v_{21}^{i}a) & v_{21}^{i}b+v_{22}^{i}c\end{smallmatrix}\right)\
\\
\left(\begin{smallmatrix}a & b \\
0 & c\end{smallmatrix}\right)
\left(\begin{smallmatrix}v_{11}^{i} & v_{12}^{i} \\
(u_{21}^{i-1},v_{21}^{i}) & v_{22}^{i}\end{smallmatrix}\right)\
=
\left(\begin{smallmatrix}av_{11}^{i}+bv_{21}^{i} & av_{12}^{i}+bv_{22}^{i} \\
(cu_{21}^{i-1},cv_{21}^{i}) & cv_{22}^{i}\end{smallmatrix}\right)\
\end{aligned}\]
for any
$\left(\begin{smallmatrix}v_{11}^{i} & v_{12}^{i} \\
(u_{21}^{i-1},v_{21}^{i}) & v_{22}^{i}\end{smallmatrix}\right)\
\in \left(\begin{smallmatrix} V^{i} & V^{i} \\
V^{i-1}\oplus V^{i} & V^{i}\end{smallmatrix}\right)\ ,
\left(\begin{smallmatrix}a & b \\
0 & c\end{smallmatrix}\right)
\in T_{2}(R)$,
and
\[\Delta^{i}=
\left(\begin{smallmatrix} d^{i} & d^{i} \\
\partial^{i} & d^{i}\end{smallmatrix}\right)\
:
\left(\begin{smallmatrix} V^{i} & V^{i} \\
V^{i-1}\oplus V^{i} & V^{i}\end{smallmatrix}\right)\
\to
\left(\begin{smallmatrix} V^{i+1} & V^{i+1} \\
V^{i}\oplus V^{i+1} & V^{i+1}\end{smallmatrix}\right)\
\]
is a differential with $\partial^{i}=\left(\begin{smallmatrix} -d^{i-1} & 1 \\
0 & d^{i}\end{smallmatrix}\right)\:V^{i-1}\oplus V^{i} \to V^{i}\oplus V^{i+1}$.
\end{proof}

\begin{exmp}\label{dualizing1}
(1) A commutative ring obviously has a two-sided injective resolution.

(2) Let $k$ be a commutative ring, $A$ a $k$-algebra which is projective as a $k$-module.
According to \cite{Mi4}, $A$ has a two-sided injective resolution.
\end{exmp}

\begin{lem}\label{st-h41}
Let $R$ be an Iwanaga-Gorenstein ring.
If $R$ has a two-sided injective resolution, then
the following hold.
\begin{enumerate}
\item
$(\cat{K}^{\infty ,\emptyset}(\Pj {R}), \cat{K}^{+,\mrm{b}}(\Pj {R}))$ is
a stable t-structure in $\cat{K}^{\infty ,\mrm{b}}(\Pj {R})$. 
\item
$(\cat{K}^{-,\mrm{b}}(\Ij {R}), \cat{K}^{\infty ,\emptyset}(\Ij {R}))$ is
a stable t-structure in $\cat{K}^{\infty ,\mrm{b}}(\Ij {R})$.
\end{enumerate}
\end{lem}

\begin{proof}
Let $V$ be a two-sided injective resolution of $R$, then $V$ is a dualizing complex.
According to \cite{IK} Theorem 4.2 and 4.8,  
$G=-\ten_{R}V:\cat{K}(\Pj R) \to \cat{K}(\Ij R)$
is a triangle equivalence. Also \cite{IK} 
Theorem 2.7, Proposition 3.4, Proposition 5.9 and 5.12 show that 
$G$ restricts to equivalences 
\[\begin{aligned}
\cat{K}^{\infty, \mrm{b}}(\Pj R) \to \cat{K}^{\infty, \mrm{b}}(\Ij R), 
& \cat{K}^{\infty, \emptyset}(\Pj R) \to \cat{K}^{\infty, \emptyset}(\Ij R), \\
\cat{K}^{+,\mrm{b}}(\Pj R) \to \cat{K}^{+,\mrm{b}}(\Ij R), 
& \cat{K}^{-,\mrm{b}}(\Pj R) \to \cat{K}^{-,\mrm{b}}(\Ij R).
\end{aligned}\]
According to \cite{IK} Corollary 5.5 and 5.12, 
$\cat{K}^{\infty, \emptyset}(\Pj R) 
$ and 
$\cat{K}^{\infty, \emptyset}(\Ij R) 
$ 
coincide with the homotopy categories of totally acyclic complexes of
projective modules and injective modules, respectively.
Applying $G$, we obtain (1) from  Proposition \ref{st-st0} and (2) from 
Proposition \ref{st-st1f}. 
\end{proof}

\begin{prop}\label{lcm} 
The following hold.
\begin{enumerate}
\rmitem{1} $\LCM R=\{Z^0(X)\ |\ X\in\cat{C} (\Pj R), \mrm{H}^i (X) =0 ~(i \in \mathbb{Z} ) \}$.
\rmitem{2} $\underline{\LCM} R\simeq\cat{K}^{\infty,\emptyset}(\Pj R)$.
\end{enumerate}
\end{prop}

\begin{proof}
By Lemma \ref{st-h41}, this proof is similar  to one of Lemma \ref{cm}.
\end{proof}

\begin{thm} \label{st-h31}
Let $R$ be an Iwanaga-Gorenstein ring with a two-sided injective resolution, then
the following hold.
\begin{enumerate}
\item
$((\cat{K}^{+,\mrm{b}}/\cat{K}^{\mrm{b}})(\Pj {R}), 
(\cat{K}^{-,\mrm{b}}/\cat{K}^{\mrm{b}})(\Pj {R}), 
\cat{K}^{\infty,\emptyset}(\Pj {R}))$ is a triangle of recollements
in $(\cat{K}^{\infty,\mrm{b}}/\cat{K}^{\mrm{b}})(\Pj {R})$.
\item
$((\cat{K}^{+,\mrm{b}}/\cat{K}^{\mrm{b}})(\Ij {R}), 
(\cat{K}^{-,\mrm{b}}/\cat{K}^{\mrm{b}})(\Ij {R}), 
\cat{K}^{\infty ,\emptyset}(\Ij {R}))$ is a triangle of recollements
in $(\cat{K}^{\infty ,\mrm{b}}/\cat{K}^{\mrm{b}})(\Ij {R})$.
\item
We have triangle equivalences
\[\begin{aligned}
(\cat{K}^{+,\mrm{b}}/\cat{K}^{\mrm{b}})(\Pj {R})\simeq 
(\cat{K}^{-,\mrm{b}}/\cat{K}^{\mrm{b}})(\Pj {R})\simeq
\cat{K}^{\infty ,\emptyset}(\Pj {R})\simeq \\
(\cat{K}^{+,\mrm{b}}/\cat{K}^{\mrm{b}})(\Ij {R})\simeq 
(\cat{K}^{-,\mrm{b}}/\cat{K}^{\mrm{b}})(\Ij {R})\simeq
\cat{K}^{\infty ,\emptyset}(\Ij {R})
\end{aligned}\]
\end{enumerate}
\end{thm}

\begin{proof} (1)(2) It is immediate from Theorem \ref{st-st1} and Lemma \ref{st-h41}.  

(3) We know the equivalence 
$\cat{K}^{\infty ,\emptyset}(\Pj  {R}) \simeq \cat{K}^{\infty ,\emptyset}(\Ij {R})$ from 
the proof of Lemma \ref {st-h41}. Other equivalences are obtained by Proposition \ref{t-st-first} (3).
\end{proof}

Once we get a triangle of recollements in $(\cat{K}^{\infty,\mrm{b}}/\cat{K}^{\mrm{b}})(\Pj {R})$, 
the remaining part is just $\LCM$-versions of the proevious results.

\begin{defn} 
\begin{enumerate}
\item 
We denote by $\underline{\Morph}^{\mrm{e}}(\LCM R)$ the full subcategory                                                                                                                        
of \\
$\Morph(\Mod R)$ consisting of objects $\alpha:X \to T$ which is surjective and $X, T \in\LCM R$.

\item
We denote by $\underline{\Morph}^{01}(\LCM R)$ (resp., $\underline{\Morph}^{10}(\LCM R)$, $\underline{\Morph}^{01}(\LCM R)$)
the full subcategory of $\underline{\Morph}^{\mrm{e}}(\LCM R)$ consisting of objects
corresponding to the objects of the form $X \to 0$ (resp., $T \xarr{1} T$,
$P \to T$ with $P$ being projective) in $\Morph ^s (\LCM R)$. 

\item
We denote by $\underline{\Morph}^{\mrm{m}}(\LCM R)$ the full subcategory                                                                                                                        
of $\Morph(\Mod R)$ consisting of objects $\alpha:Z \to X$ which is injective and $Z, X, \opn{cok} \alpha \in\LCM R$.

\end{enumerate}
\end{defn}

\begin{prop}\label{BCMT2}
\begin{itemize}
\item[(1)] The equivalence in Proposition \ref{T2} induces an equivalence
\[\Morph^{\mrm{m}}(\LCM R)\simeq\LCM\opn{T}_2(R).\]
\item[(2)] $\alpha\mapsto\opn{cok}\alpha$ and $\alpha\mapsto\opn{ker}\alpha$ give mutually quasi-inverse equivalences
\[\xymatrix{\Morph^{\mrm{m}}(\LCM R)\ar@<-2ex>[r]^{\opn{cok}}&\Morph^{\mrm{e}}(\LCM R)\ar@<-1ex>[l]_{\opn{ker}}.}\]
\end{itemize}
\end{prop}

\begin{proof}
By Lemmas \ref{Lem:2Ijres} and \ref{lcm}, this proof is similar to one of Proposition \ref{CMT2}.
\end{proof}

We also have the following $\LCM$-version of Theorem \ref{st-st2}.

\begin{thm} \label{Bst-st2}
Let $R$ be an Iwanaga-Gorenstein ring with a two-sided injective resolution.
Then we have the following.
\begin{itemize}
\item[(1)] $(\underline{\Morph}^{01}(\LCM R),\underline{\Morph}^{10}(\LCM R),
\underline{\Morph}^{11}(\LCM R))$ is a triangle of recollements in $\underline{\Morph}^{\mrm{e}}(\LCM R)$.
\item[(2)] There is a recollement
\[\xymatrix{
\underline{\Morph}^{10}(\LCM R)\ar@<-1ex>[r]^{i_{*}}
& \hspace{15pt}\underline{\Morph}^{\mrm{e}}(\LCM R)\hspace{15pt}
\ar@<-2ex>[l]_{i^{*}} \ar@<2ex>[l]^{i^{!}} \ar@<-1ex>[r]^{j^{*}} 
& \hspace{6pt}\underline{\Morph}^{\mrm{e}}(\LCM R)/\underline{\Morph}^{10}(\LCM R)
\ar@<-2ex>[l]_{j_{!}} \ar@<2ex>[l]^{j_{*}}
}\]
such that the essential image $\opn{Im}j_{!}$ (resp., $\opn{Im}j_{*}$)
of the functor $j_{!}$ (resp., $j_{*}$) is $\underline{\Morph}^{01}(\LCM R)$ 
(resp., $\underline{\Morph}^{11}(\LCM R)$).
\item[(3)] We have equivalences 
\[\underline{\LCM} R \simeq \underline{\Morph}^{10}(\LCM R)
\simeq \underline{\Morph}^{11}(\LCM R)
\simeq \underline{\Morph}^{01}(\LCM R).\]
\end{itemize}
\end{thm} 

\begin{proof}
The same as the proof of Theorem \ref{st-st2}.
\end{proof}

In the same way as in section \ref{trieqrecoll}, we have a functor 
$F: \Morph^{\mrm{e}}(\LCM R) \to \cat{K}^{\infty, \mrm{b}}(\Pj {R})$, 
and a triangle functor
$\underline{F} : \underline{\Morph}^{\mrm{e}}(\LCM R) \to 
({\sf K}^{\infty,\mrm{b}}/ {\sf K}^{\mrm{b}})(\Pj {R})$.

\begin{lem} \label{Bst-eq2bis}
The functor 
$\underline{F}\vert_{\underline{\Morph}^{11}(\LCM R)}:\underline{\Morph}^{11}(\LCM R) \to 
\cat{K}^{\infty,\emptyset}(\Pj {R})$
is a triangle equivalence.
\end{lem}

\begin{proof}
By Lemma \ref{lcm}, this proof is similar to one of Lemma \ref{st-eq2bis}.
\end{proof}

\noindent {\it Proof of Theorem \ref{Bmth}.}~
By Theorem \ref{Bst-st2}, we have only to show that $\underline{F}$ is an triangle equivalence. 
It is easy to see that $\underline{F}$ sends a triangle
\[(\underline{\Morph}^{01}(\LCM R),\underline{\Morph}^{10}(\LCM R),
\underline{\Morph}^{11}(\LCM R))\]
of recollements to a triangle 
\[((\cat{K}^{+,\mrm{b}}/\cat{K}^{\mrm{b}})(\Pj {R}), 
(\cat{K}^{-,\mrm{b}}/\cat{K}^{\mrm{b}})(\Pj {R}), 
\cat{K}^{\infty ,\emptyset}(\Pj {R}))\]
of recollements.
The equivalence $\underline{F}\vert_{\underline{\Morph}^{11}(\LCM R)}:\underline{\Morph}^{11}(\LCM R) \to 
\cat{K}^{\infty,\emptyset}(\Pj {R})$ in Lemma \ref{Bst-eq2bis} yields the desired equivalence 
$\underline{\Morph}^{\mrm{e}}(\LCM R)\simeq 
({\sf K}^{\infty,\mrm{b}}/ {\sf K}^{\mrm{b}})(\Pj {R})$
by Propositions \ref{p20Apr30}. 
\qed

\begin{cor} \label{Bmcor}
Let $R$ be an Iwanaga-Gorenstein ring,
If $R$ has a finite two-sided injective resolution, then
there is a triangle equivalence from $\underline{\LCM}\opn{T}_2(R)$
to $(\cat{K}^{\infty,\mrm{b}}/\cat{K}^{\mrm{b}})(\Ij {R})$.
\end{cor}

\begin{proof}
We have a triangle equivalence 
$(\cat{K}^{\infty,\mrm{b}}/\cat{K}^{\mrm{b}})(\Pj {R})\simeq
(\cat{K}^{\infty,\mrm{b}}/\cat{K}^{\mrm{b}})(\Ij {R})$ 
by the proof of Lemma \ref{st-h41}. 
So we get the conclusion by Theorem \ref{Bmth}.
\end{proof}

\section{Algebraic triangulated categories and triangles of recollement}\label{agtri}

In this section, we see that results of previous sections also hold in any Frobenius category.

\begin{defn}\label{Fbtac}
Let $\mcal{F}$ be a Frobenius category with the additive subcategory $\mcal{P}$ of 
projective-injective objects.
We denote by the full subcategory $\cat{K}^{\infty, \emptyset}(\mcal{P}, \mcal{F})$
of $\cat{K}(\mcal{P})$ consisting of complexes $X$ such that $Z^i(X)$ exists and
$0 \to \opn{Z}^i(X) \to X^i \to \opn{Z}^{i+1}(X) \to 0$ is an exact sequence of $\mcal{F}$
for any $i \in \mathbb{Z}$.
And we denote by $\underline{\mcal{F}}$ the stable category of $\mcal{F}$ by $\mcal{P}$.
\end{defn}

\begin{rem}
For an exact category $\mcal{E}$, Neeman defined the homotopy category $\cat{A}(\mcal{E})$ of complexes $X$ such that
$0 \to \opn{Z}^i(X) \to X^i \to \opn{Z}^{i+1}(X) \to 0$ is an exact sequence of $\mcal{E}$ (\cite{Ne1}). 
He showed that $\cat{A}(\mcal{E})$ is a full triangulated subcategory of $\cat{K}(\mcal{E})$.
Then the above $\cat{K}^{\infty, \emptyset}(\mcal{P}, \mcal{F})$ is a full triangulated subcategory of $\cat{A}(\mcal{F})$.
\end{rem}

\begin{prop} \label{AgTac}
Let $\mcal{F}$ be a Frobenius category with the additive subcategory $\mcal{P}$ of 
projective-injective objects.
Then $\cat{K}^{\infty, \emptyset}(\mcal{P}, \mcal{F})$ is a triangulated category which is equivalent to
$\underline{\mcal{F}}$.
\end{prop}

\begin{proof}
By the above remark, $\cat{K}^{\infty, \emptyset}(\mcal{P}, \mcal{F})$ is a full triangulated subcategory of 
$\cat{K}(\mcal{P})$.
By the same reason as Proposition \ref{cm},
the functor $Z^0:\cat{K}^{\infty, \emptyset}(\mcal{P}, \mcal{F}) \to \underline{\mcal{F}}$
is a triangulated equivalence.
\end{proof}

\begin{defn}\label{ghtp01}
Let $\mcal{F}$ be a Frobenius category with the additive subcategory $\mcal{P}$ of 
projective-injective objects.
We denote by the full subcategory $\cat{K}^{\infty, \mrm{b}}(\mcal{P}, \mcal{F})$ of $\cat{K}(\mcal{P})$ 
consisting of complexes $X$ satisfying $\tau_{\leq m}X=\tau_{\leq m}Y$ 
and $\tau_{\geq n}X=\tau_{\geq n}Y'$ for some integers $m\leq n$ and 
some complexes $Y, Y' \in \cat{K}^{\infty, \emptyset}(\mcal{P}, \mcal{F})$.
Let $\cat{K}^{+,\mrm{b}}(\mcal{P}, \mcal{F})=
\cat{K}^{+}(\mcal{P})\cap\cat{K}^{\infty, \mrm{b}}(\mcal{P}, \mcal{F})$ and
$\cat{K}^{-,\mrm{b}}(\mcal{P}, \mcal{F})=
\cat{K}^{-}(\mcal{P})\cap\cat{K}^{\infty, \mrm{b}}(\mcal{P}, \mcal{F})$.
\end{defn}

\begin{lem}\label{gst-st0}
Let $\mcal{F}$ be a Frobenius category with the additive subcategory $\mcal{P}$ of 
projective-injective objects.
Then the following hold.
\begin{enumerate}
\item  A pair $(\cat{K}^{-, \mrm{b}}(\mcal{P},\mcal{F}), \cat{K}^{\infty, \emptyset}(\mcal{P}, \mcal{F}))$ 
is a stable t-structure in $\cat{K}^{\infty, \mrm{b}}(\mcal{P}, \mcal{F})$.
\item  A pair $(\cat{K}^{\infty, \emptyset}(\mcal{P}, \mcal{F}), \cat{K}^{+, \mrm{b}}(\mcal{P},\mcal{F}))$
is a stable t-structure in $\cat{K}^{\infty, \mrm{b}}(\mcal{P}, \mcal{F})$.
\end{enumerate}
\end{lem}

\begin{proof}
Since $\Hom_{\mcal{P}}(B,Y)$ is acyclic for any $B \in \mcal{P}$ and $Y\in\cat{K}^{\infty, \emptyset}(\mcal{P}, \mcal{F})$, 
it is easy to see that 
\[
\Hom_{\cat{K}(\mcal{P})}(\cat{K}^{-, \mrm{b}}(\mcal{P}, \mcal{F}),\cat{K}^{\infty, \emptyset}(\mcal{P}, \mcal{F}))=0.
\]
For $X \in \cat{K}^{\infty, \mrm{b}}(\mcal{P}, \mcal{F})$,
there are an integer $n$ and a complex $Y \in  \cat{K}^{\infty, \emptyset}(\mcal{P}, \mcal{F})$ such that 
$\tau_{\geq n}X=\tau_{\geq n}Y$.
Then we have a morphism $f:X \to Y$ such that $f^{i}=1_{X^i}$ for any $i \geq n$
and a triangle in $\cat{K}^{\infty, \mrm{b}}(\mcal{P}, \mcal{F})$:
\[
Z \to X \xarr{f} Y \to \Sigma Z .
\]
It is easy to see that $Z$ is isomorphic to some complex of $\cat{K}^{-, \mrm{b}}(\mcal{P}, \mcal{F})$
in $\cat{K}^{\infty,  \mrm{b}}(\mcal{P}, \mcal{F})$.
\end{proof}

We denote by $(\cat{K}^{*,\mrm{b}}/\cat{K}^{\mrm{b}})(\mcal{P}, \mcal{F})$ 
the quotient category of $\cat{K}^{*,\mrm{b}}(\mcal{P}, \mcal{F})$ by $\cat{K}^{\mrm{b}}(\mcal{P})$.

\begin{thm}\label{gst-h3}
Let $\mcal{F}$ be a Frobenius category with the additive subcategory $\mcal{P}$ of 
projective-injective objects.
Then
$((\cat{K}^{+,\mrm{b}}/\cat{K}^{\mrm{b}})(\mcal{P}, \mcal{F}), 
(\cat{K}^{-,\mrm{b}}/\cat{K}^{\mrm{b}})(\mcal{P}, \mcal{F}),\\
\cat{K}^{\infty, \emptyset}(\mcal{P}, \mcal{F}))$
is a triangle of recollements
in $(\cat{K}^{\infty,\mrm{b}}/\cat{K}^{\mrm{b}})(\mcal{P}, \mcal{F})$.
In particular, we have triangle equivalences
\[
(\cat{K}^{+,\mrm{b}}/\cat{K}^{\mrm{b}})(\mcal{P}, \mcal{F})\simeq 
(\cat{K}^{-,\mrm{b}}/\cat{K}^{\mrm{b}})(\mcal{P}, \mcal{F})\simeq
 \cat{K}^{\infty, \emptyset}(\mcal{P}, \mcal{F})
\]
\end{thm}

\begin{proof}
As well as Proposition \ref{pm}, it is easy to see that
$((\cat{K}^{+,\mrm{b}}/\cat{K}^{\mrm{b}})(\mcal{P}, \mcal{F}), \\
(\cat{K}^{-,\mrm{b}}/\cat{K}^{\mrm{b}})(\mcal{P}, \mcal{F}))$ is a stable
t-structure.
According to Lemma \ref{gst-st0} and Proposition \ref{st-stK},
$((\cat{K}^{-,\mrm{b}}/\cat{K}^{\mrm{b}})(\mcal{P}, \mcal{F}),  \cat{K}^{\infty, \emptyset}(\mcal{P}, \mcal{F}))$
and
$(\cat{K}^{\infty, \emptyset}(\mcal{P}, \mcal{F}), (\cat{K}^{+,\mrm{b}}/\cat{K}^{\mrm{b}})(\mcal{P}, \mcal{F}))$
are stable t-structures.
\end{proof}

\begin{defn}
Let $\mcal{F}$ be a Frobenius category with the subcategory $\mcal{P}$ of \\
projective-injective objects.
We define the category $\Morph^{\mrm{e}}(\mcal{F})$ 
(resp., $\Morph^{\mrm{m}}(\mcal{F})$) as follows.
\begin{itemize}
\item An object is an admissible epimorphism (resp., an admissible monomorphism) 
$\alpha: X \to Y$.
\item A morphism from $\alpha: X \to Y$ to $\beta: X' \to Y'$ is a pair $(f,g)$ of morphisms 
$f:X \to X'$ and $g:T \to T'$ such that $g\alpha = \beta f$.
\end{itemize}
We denote by $\Pj(\Morph^{\mrm{e}}(\mcal{F}))$ (resp., $\Pj(\Morph^{\mrm{m}}(\mcal{F}))$)
the full subcategory of $\Morph^{\mrm{e}}(\mcal{F})$ (resp., $\Morph^{\mrm{m}}(\mcal{F})$)
consisting objects $X \to Y$ with $X, Y \in \mcal{P}$, and denote by
$\underline{\Morph}^{\mrm{e}}(\mcal{F})$ (resp., $\underline{\Morph}^{\mrm{m}}(\mcal{F})$)
the stable category of $\underline{\Morph}^{\mrm{e}}(\mcal{F})$
(resp., $\underline{\Morph}^{\mrm{m}}(\mcal{F})$) by 
$\Morph^{\mrm{e}}(\mcal{F})$ (resp., $\Morph^{\mrm{m}}(\mcal{F})$).
\end{defn}

\begin{prop}
Let $\mcal{F}$ be a Frobenius category with the subcategory $\mcal{P}$ of \\
projective-injective objects.
Then the following hold.
\begin{enumerate}
\item  $\Morph^{\mrm{e}}(\mcal{F})$ is a Frobenius category, and 
$\underline{\Morph}^{\mrm{e}}(\mcal{F})$ is a triangulated category.
\item  $\Morph^{\mrm{m}}(\mcal{F})$ is a Frobenius category, and 
$\underline{\Morph}^{\mrm{m}}(\mcal{F})$ is a triangulated category.
\item  $\Morph^{\mrm{e}}(\mcal{F})$ is equivalent to  $\Morph^{\mrm{m}}(\mcal{F})$ as a Frobenius category,
and $\underline{\Morph}^{\mrm{e}}(\mcal{F})$ is triangle equivalent to
$\underline{\Morph}^{\mrm{m}}(\mcal{F})$.
\end{enumerate}
\end{prop}

\begin{proof}
As well as Proposition \ref{CMT2}.
\end{proof}

\begin{thm} \label{gst-st2}
Let $\mcal{F}$ be a Frobenius category with the subcategory $\mcal{P}$ of 
projective-injective objects.
Let $\underline{\Morph}^{10}(\mcal{F})$ (resp., $\underline{\Morph}^{11}(\mcal{F})$, 
$\underline{\Morph}^{01}(\mcal{F})$)
the full subcategory of $\underline{\Morph}^{\mrm{e}}(\mcal{F})$ consisting of objects
corresponding to the objects of the form $X \to 0$ (resp., $T \xarr{1} T$,
$P \to T$ with $P\in \mcal{P}$).
Then 
$(\underline{\Morph}^{10}(\mcal{F}),\underline{\Morph}^{11}(\mcal{F}),
\underline{\Morph}^{01}(\mcal{F}))$ is a triangle of recollements
in $\underline{\Morph}^{\mrm{e}}(\mcal{F})$.
\end{thm}

\begin{proof}
Similar to the proof of Theorem \ref{st-st2}.
\end{proof}

\begin{thm}\label{gmth}
Let $\mcal{F}$ be a Frobenius category with the additive subcategory $\mcal{P}$ of 
projective-injective objects.
Then $\underline{\Morph}^{\mrm{e}}(\mcal{F})$ is triangle equivalent to \\
$(\cat{K}^{\infty,\mrm{b}}/\cat{K}^{\mrm{b}})(\mcal{P}, \mcal{F})$.
\end{thm}

\begin{proof}
For $\alpha:X \to Y \in \Morph^{\mrm{e}}(\mcal{F})$ 
we can construct a complex $F_{\alpha}$ as the same construction in Section \ref{trieqrecoll}, and then
a functor $\underline{F}:\underline{\Morph}^{\mrm{e}}(\mcal{F}) \to 
(\cat{K}^{\infty,\mrm{b}}/\cat{K}^{\mrm{b}})(\mcal{P}, \mcal{F})$.
Moreover, it is easy to see that $\underline{F}$ sends a triangle
$(\underline{\Morph}^{10}(\mcal{F}),\underline{\Morph}^{11}(\mcal{F}),
\underline{\Morph}^{01}(\mcal{F}))$
of recollements to a triangle 
$((\cat{K}^{-,\mrm{b}}/\cat{K}^{\mrm{b}})(\mcal{P}, \mcal{F}),  \cat{K}^{\infty, \emptyset}(\mcal{P}, \mcal{F}),
(\cat{K}^{+,\mrm{b}}/\cat{K}^{\mrm{b}})(\mcal{P}, \mcal{F}))$
of recollements.
As well as Lemma \ref{st-eq2bis}, $F\vert_{\underline{\Morph}^{11}(\mcal{F})}:
\underline{\Morph}^{11}(\mcal{F}) \to  \cat{K}^{\infty, \emptyset}(\mcal{P}, \mcal{F})$ is a triangle equivalence.
By Proposition \ref{p20Apr30}, we have the conclusion.
\end{proof}

\begin{exmp}\label{TotAc}
Let $\mcal{A}$ be an abelian category, and let $\mcal{P}$ be its additive full subcategory satisfying 
that $\Ext^{1}_{\mcal{A}}(X,Y)=0$ for any $X, Y \in \mcal{P}$.
A complex $X\in \cat{K}(\mcal{P})$ is called \emph{totally acyclic}
if $\Hom_{\mcal{B}}(B,X)$ and $\Hom_{\mcal{B}}(X,B)$ are both acyclic for any $B \in \mcal{P}$.
Let $\mcal{G}(\mcal{P})$ be the full subcategory of $\mcal{A}$ consisting objects which are isomorphic to  the
$0$-cycle $\opn{Z}^0(X)$ of some totally acyclic complex $X \in \cat{K}(\mcal{P})$.
Then $\mcal{G}(\mcal{P})$ is a Frobenius category with the additive subcategory $\mcal{P}$ of 
projective-injective objects.
Therefore $(\cat{K}^{\infty,\mrm{b}}/\cat{K}^{\mrm{b}})(\mcal{P}, \mcal{G}(\mcal{P}))$ has a
triangle of recollements.
In particular, for a commutative noetherian ring $R$ with a dualizing complex, 
$\cat{K}^{\infty,\mrm{b}}(\Pj R, \mcal{G}(\Pj R))$ is called the symmetric Auslander category, 
a triangle of recollements in $(\cat{K}^{\infty,\mrm{b}}/\cat{K}^{\mrm{b}})(\Pj R, \mcal{G}(\Pj R))$
are studied in \cite{JK}.
\end{exmp}

\begin{exmp}\label{ATFR4} 
Let $k$ be a field, and let $A$ be the $k$-algebra defined by the following quiver with relations:
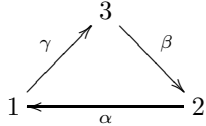
\[
\xymatrix{
& 3 \ar[dr]^{\beta}  \\
1 \ar[ur]^{\gamma} & & 2 \ar[ll]^{\alpha}
}
\]
with $\alpha\beta\gamma=\beta\gamma\alpha = \gamma\alpha\beta=0$.
Let $\mcal{F}$ be the additive subcategory of $\fmod A$ generated by finitely generated projective
right $A$-modules, $e_1A/e_1J_A$ and $e_2A/\cat{Soc}(e_2A)$, where $J_A$ is the Jacobson radical right $A$-modules, where $J_A$ is the Jacobson radical and
$e_i$ is the idempotent corresponding a vertex $i$.
Then $\mcal{F}$ is a Frobenius category with the subcategory $\pj A$ of finitely generated 
projective-injective $A$-modules.
Then we have the following equivalences
\[
(\cat{K}^{+,\mrm{b}}/\cat{K}^{\mrm{b}})(\pj A, \mcal{F})\simeq
(\cat{K}^{-,\mrm{b}}/\cat{K}^{\mrm{b}})(\pj A, \mcal{F})\simeq
\cat{K}^{\infty,\emptyset}(\pj A, \mcal{F})\simeq
\underline{\fmod} B, 
\]
\[
(\cat{K}^{\infty,\mrm{b}}/\cat{K}^{\mrm{b}})(\pj A, \mcal{F})\simeq
\underline{\fmod} \opn{T}_2(B).
\]
where $B$ is  the  algebra defined by the following quiver with relations:
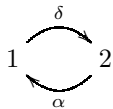
\[
\xymatrix{
1 \ar@/^1pc/[r]^{\delta} & 2 \ar@/^1pc/[l]^{\alpha}
}
\]
with $\alpha\delta=\delta\alpha=0$.
On the other hand, we have 
\[\begin{aligned}
(\cat{K}^{+,\mrm{b}}/\cat{K}^{\mrm{b}})(\pj A, \fmod A) &\simeq
(\cat{K}^{-,\mrm{b}}/\cat{K}^{\mrm{b}})(\pj A, \fmod A)\simeq
\cat{K}^{\infty,\emptyset}(\pj A, \fmod A) \\
& \simeq \underline{\fmod} A, 
\end{aligned}\]
\[
(\cat{K}^{\infty,\mrm{b}}/\cat{K}^{\mrm{b}})(\pj A, \fmod A)\simeq
\underline{\fmod} \opn{T}_2(A).
\]
Thus the category $\cat{K}^{\infty, \mrm{b}}(\mcal{P}, \mcal{F})$ depends not only on 
$\mcal{P}$ but also on $\mcal{F}$. 

\end{exmp}


\end{document}